\documentclass[11pt]{amsart}
\usepackage{graphicx}
\usepackage{caption}
\usepackage{subcaption}
\usepackage{amssymb}
\usepackage{amsmath}
\usepackage{amsthm}
\usepackage{bm}
\usepackage{placeins}
\usepackage{verbatim}
\usepackage{multirow}
\usepackage{longtable}
\usepackage{url}
\usepackage[all]{xy}
\usepackage{url}
\usepackage{array,booktabs}
\usepackage[width=0.9\textwidth]{caption}
\newtheorem{tw}{Theorem}[section]
\newtheorem{prop}[tw]{Proposition}
\newtheorem{lem}[tw]{Lemma}
\newtheorem{wn}[tw]{Corollary}

\theoremstyle{remark}
\newtheorem{uw}[tw]{Remark}

\theoremstyle{definition}

\newcommand{\bez}{\setminus}

\newcommand{\ro}{\varrho}

\newcommand{\gen}[1]{\langle #1 \rangle}
\newcommand{\biggen}[1]{\Big\langle #1 \Big\rangle}
\newcommand{\map}[3]{#1\colon #2\to #3}
\newcommand{\field}[1]{\mathbb{#1}}
\newcommand{\zz}{\field{Z}}

\newcommand{\st}{\;|\;}
\newcommand{\bigst}{\;\Big|\;}

\newcommand{\CondB}{(D3B)}

\newcommand{\TypeA}{A}
\newcommand{\TypeB}{B}

\begin{document}

\numberwithin{equation}{section}
\title{Roots of Dehn twists on nonorientable surfaces}

\author{Anna Parlak \hspace{1em}  Micha\l $\ $Stukow}


\address[]{
Institute of Mathematics, Faculty of Mathematics, Physics and Informatics, University of Gda\'nsk, 80-308 Gda\'nsk, Poland
}

\thanks{Both authors are supported by grant 2015/17/B/ST1/03235 of National Science Centre, Poland.}

\email{anna.parlak@gmail.com, trojkat@mat.ug.edu.pl}


\begin{abstract}
Margalit and Schleimer observed that Dehn twists on orientable surfaces have nontrivial roots.  We investigate the problem of roots of a Dehn twist $t_c$ about a nonseparating circle $c$ in the mapping class group $\mathcal{M}(N_{g})$ of a nonorientable surface $N_g$ of genus $g$. We explore the existence of roots and, following the work of McCullough, Rajeevsarathy and Monden, give a simple arithmetic description of their conjugacy classes. We also study roots of maximal degree and prove that if we fix an odd integer $n>1$, then for each sufficiently large $g$, 
$t_c$ has a root of degree $n$ in $\mathcal{M}(N_{g})$. 
Moreover, for any possible degree $n$ we provide explicit expressions for a particular type of roots of Dehn twists about nonseparating circles in $N_g$. 
\end{abstract}


\maketitle%

	\section{Introduction}
	Let $N_{g, n}$ be a connected nonorientable surface of genus $g$ with $n$ boundary components, that is a surface obtained from a connected sum of $g$ projective planes $N_{g}$ by removing $n$ open disks. If $n$ equals zero, we omit it in notation. 
	The \emph{mapping class group} $\mathcal{M}(N_{g, n})$ consists of isotopy classes of self-homeomorphisms $h: N_{g,n}\rightarrow N_{g,n}$ such that $h$ is the identity on each boundary component.  The mapping class group $\mathcal{M}(S_{g,n})$ of a connected orientable surface of genus $g$ with $n$ boundary components is defined analogously, but we consider only orientation-preserving maps. 
	By abuse of notation, we identify a homeomorphism with its isotopy class. 
	
	Margalit and Schleimer showed that every Dehn twist about a nonseparating circle in $\mathcal{M}(S_{g+1})$, $g \geq 1$, has a root of degree $2g+1$ \cite{MargSchleim}. A natural question, arising as a consequence of this result, is a question about other possible degrees  of such roots.
	This research has been pursued by McCullough and Rajeevsarathy \cite{McCullough2010}. Using the theory of finite group actions on surfaces \cite[Chapter 13]{Thurston_pre} they derived necessary and sufficient conditions for the existence of a root of degree $n$ of a Dehn twist about a nonseparating circle in $\mathcal{M}(S_{g+1})$. They also gave an arithmetic description of conjugacy classes of such roots. Our aim is to conduct an analogous investigation in the nonorientable case, which is a natural continuation of our previous work in \cite{ParlakStukowRoots}.
	
	In Section \ref{sec:pre} we recollect some basic facts about Dehn twists on nonorientable surfaces. We prove that if a root of a Dehn twist about a nonseparating circle in $\mathcal{M}(N_g)$ exists, it must be of odd degree (Proposition \ref{Prop:odd:degree}). We also formulate and prove the 'trident relation' (Proposition \ref{tw:trident}) which is a generalization of both the chain and the star relation in $\mathcal{M}(S_g)$. We will make an extensive use of this relation in the final section of this paper.
	
	In Section \ref{sec:data}, following the work of McCullough and Rajeevsarathy \cite{McCullough2010}, we define a \emph{data set of degree $n$} which encodes the information required to describe a root of degree $n$ (Theorem \ref{tw:orb}). One of the main differences between the orientable and the nonorientable case is that when $g$ is even not all Dehn twists about nonseparating circles in $N_g$ are conjugate in $\mathcal{M}(N_g)$, thus we distinguish two types of data sets. The distinction is based on whether $N_g\bez c$ is orientable or not. 
	
	In Section \ref{sec:degrees} we study possible degrees of roots. In the orientable case it was proved \cite{McCullough2010, Monden-roots} that roots constructed by Margalit and Schleimer \cite{MargSchleim} were of maximal degree. In the case of nonorientable surfaces this problem is a bit more messy. For some genera $g$ we give a concrete information about the maximal possible degree of a root of a Dehn twist in $\mathcal{M}(N_g)$ and for others we evaluate upper and lower bounds for a possible maximal degree (Theorems \ref{tw:all:degs:a} and \ref{tw:max:b}). 
	
	In Section \ref{sec:primary} we define \emph{primary roots} of Dehn twists about nonseparating circles and we prove that if $g$ is large enough with respect to the fixed odd integer $n\geq 3$, then a Dehn twist about a nonseparating circle has a primary root of degree $n$ in $\mathcal{M}(N_g)$ (Theorems \ref{tw:allrots:a} and \ref{tw:allrots:b}). Primary roots are also of special interest in Section \ref{sec:geom:express}, where we give explicit expressions for a family of examples of such roots.  
	
%
	\section{Preliminaries}\label{sec:pre}
	We represent $g$ crosscaps on  $N_{g,n}$ by shaded disks (Figure \ref{fig_surface}) and enumerate them from 1 to $g$. 
	\begin{figure}[h]
		\begin{center}
			\includegraphics[width=0.74\textwidth]{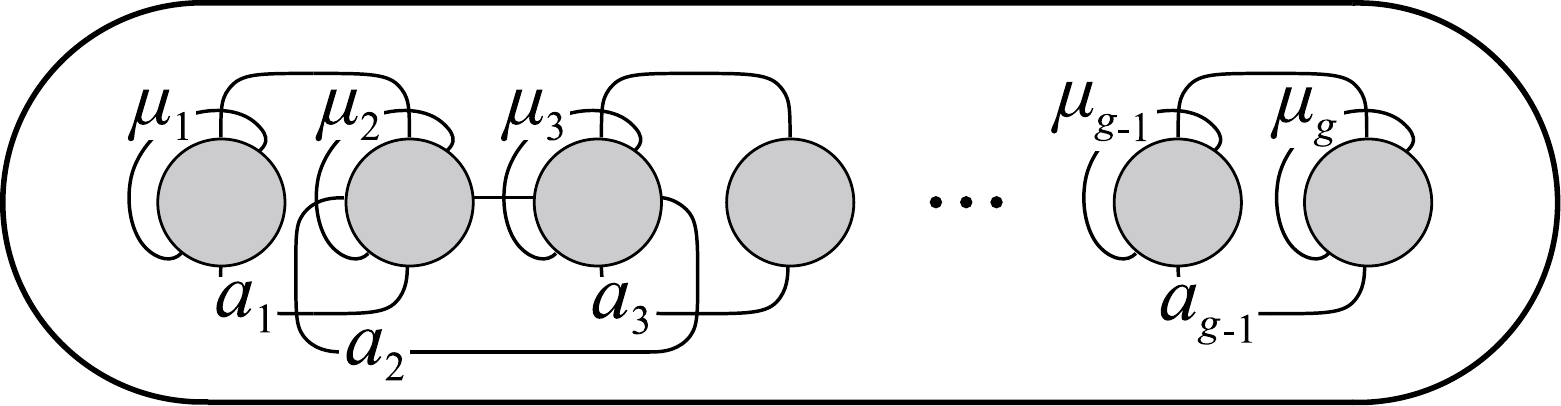}
			\caption{A nonorientable surface $N_g$.}
			\label{fig_surface} 
		\end{center}
	\end{figure}	
	By a \emph{circle} in $N_{g,n}$ we mean a simple closed curve in $N_{g,n}$ that is disjoint from the boundary $\partial N_{g,n}$. We identify a circle with its isotopy class. According to whether a regular neighbourhood of a circle is an annulus or a M\"obius strip, we call the circle \emph{two-sided} or \emph{one-sided} respectively. We say that a circle is \emph{generic} if it does not bound a disk nor a M\"obius strip. It is known (Corollary 4.5 of \cite{Stukow_twist}) that if $c$ is a generic two-sided circle on $N_g$, $g\geq 3$, then the Dehn twist $t_c$ is of infinite order in $\mathcal{M}(N_{g})$. On the other hand, if $g\leq 2$, then $\mathcal{M}(N_{g})$ is finite, hence we restrict our attention to the case $g\geq 3$ throughout the paper.
	On oriented surfaces we assume that Dehn twists twist to the left. On nonorientable surfaces there is no canonical choice of the orientation in a regular neighbourhood $S_a$ of a two-sided circle $a$, but after choosing one, we can define the \emph{Dehn twist} about $a$ in the usual way.    
	We indicate the chosen local  orientation by arrows on the two sides of the circle  which  point the direction of $t_a$, as it is presented in Figure \ref{fig:twist}. 
	\begin{figure}[h]
		\begin{center}
			\includegraphics[width=0.5\textwidth]{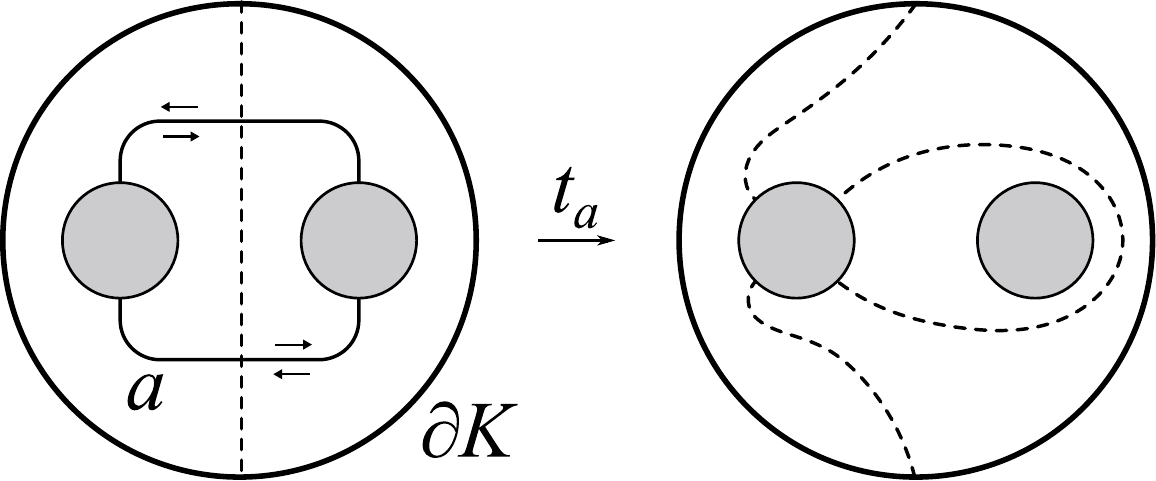}
			\caption{The Dehn twist on $K = N_{2,1}$ about the two-sided circle~$a$ passing through two crosscaps.}\label{fig:twist} 
		\end{center}
	\end{figure} 
	
	
	By $a_i$, $i = 1, \ldots, g-1$ we denote the two-sided circle passing once through each of the crosscaps labelled with $i$ and $i+1$ (Figure \ref{fig_surface}),  by $\mu_{j}$, $j = 1, \ldots, g$ we denote the one-sided circle passing once through the crosscap labelled with $j$. 
	If $g$ is even, then the circle $b$ passing once through each of the $g$ crosscaps on $N_g$ is two-sided and $N_g \bez b$ is orientable.

	If $c$ is a two-sided circle in $N_g$ such that $N_g \bez c=N_{g-2,2}$, then the Dehn twist $t_c$ is conjugate to $t_{a_1}$ (with some choice of orientations). We say that such a twist is \emph{of type A}. When $g$ is even, then there exists a second conjugacy class of twists about nonseparating circles, consisting of twists $t_c$ about circles $c$ for which $N_g \bez c = S_{\frac{g-2}{2},2}$. We say that such twists are \emph{of type B}. Clearly, $t_b$ is of type~B.

	
	Let us recall, that there exists a surjective representation \cite{Gadgil,Pinkal}
	\[\map{\psi}{\mathcal{M}(N_g)}{{\rm O}(H_1(N_g;\zz_2))},\]
	where $H_1(N_g;\zz_2)$ is a linear space with basis $\lbrack \mu_1 \rbrack, \ldots, \lbrack \mu_g \rbrack$, and ${\rm O}(H_1(N_g;\zz_2))$ are these linear automorphisms of $H_1(N_g;\zz_2)$ which preserve a $\zz_2$-valued intersection pairing $\langle\,,\rangle$ defined by $\langle\lbrack\mu_i\rbrack,\lbrack\mu_j\rbrack\rangle=\delta_{ij}$, $1\leqslant i,j\leqslant g$. 
	In particular, for any $f$ in $\mathcal{M}(N_g)$ 
		\begin{equation} \label{rep:hom:bi}
		\psi(f)^T\psi(f)=\psi(f)^TI_g\psi(f)=I_g\mod 2.
		\end{equation}
	The matrices corresponding to Dehn twists $t_{a_1}$ and $t_b$ (with respect to the basis $\lbrack \mu_1 \rbrack, \ldots, \lbrack \mu_g \rbrack$) are as follows \cite{Stukow_homolTopApp}:	
	\begin{displaymath}
		\psi(t_{a_1}) =  
		\begin{bmatrix}
			0 & 1  \\
			1 & 0  \\
		\end{bmatrix}
		\oplus I_{g-2}
	\end{displaymath}
	\begin{displaymath}
		\psi(t_{b}) = I_{g} + 
		\begin{bmatrix}
			1 & 1 & 1 & \ldots & 1  \\
			1 & 1 & 1 & \ldots & 1  \\
			\vdots & \vdots & \vdots & \ddots & \vdots\\
			1 & 1 & 1 & \ldots & 1  \\
		\end{bmatrix}_{g} 
	\end{displaymath}
	
	where $I_k$ is the identity matrix of rank $k$.
	
	The following proposition and its proof is a nonorientable version of the result of Monden (Section 5 of \cite{Monden-roots}). 
	\begin{prop}\label{Prop:odd:degree}
		Let $c$ be a nonseparating two-sided circle in a nonorientable surface $N_g$, $g\geq 3$, and let $f\in \mathcal{M}(N_g)$ be such that $f^n=t_c$ for some positive integer $n$. Then $n$ is odd.
	\end{prop}
	\begin{proof}
		It is enough to show that there are no roots of degree 2 of Dehn twists  $t_{a_1}$ and $t_b$ (the latter occurs only when $g$ is even).  
		
		Let $A = \psi(t_{a_1})$ and denote the corresponding matrix of a square root of $t_{a_1}$ by $P = (p_{ij})_{i,j=1, \ldots, g}$. By relation \eqref{rep:hom:bi}
		\[AP^T=P^2P^{-1}=P \mod 2.\]
		Since $A$ is a permutation matrix that interchanges the first two rows of $P^T$, the above equality implies that 
		\begin{enumerate}
			\item $p_{jk}=p_{kj}$ if $\{j,k\}\neq \{1,2\}$,
			\item $p_{12}=p_{11}=p_{21}$.
 		\end{enumerate}
		In particular, $P$ is symmetric. But then
		\[A=P^2=P\cdot P=P\cdot P^T=P\cdot P^{-1}=I_g\mod 2\]
		which is a contradiction.
					
		
		Let $B = \psi(t_b)$ and denote by $R = (r_{ij})_{i,j = 1, \ldots, g}$ the corresponding matrix of the square root of $t_b$. As before we get 
		\[BR^T = R \mod 2.\]
		If we compare the entries on the diagonals of both sides, we get
		$\sum_{k=1}^g r_{jk}=0$ for $j=1,\ldots,g$ (that is each column of $R^T$ sums up to 0). But this implies that 
		\[R=BR^T=\left(I_g+\begin{bmatrix}
			1 & 1  & \ldots & 1  \\			
			\vdots &  \vdots & \ddots & \vdots\\
			1 & 1 & \ldots & 1  \\
		\end{bmatrix}_{g}\right)R^T=I_gR^T + \begin{bmatrix}
		0 & 0  & \ldots & 0  \\			
		\vdots &  \vdots & \ddots & \vdots\\
		0 & 0 & \ldots & 0  \\
		\end{bmatrix}_{g}=R^T.\]
		Thus $R$ is symmetric and as in the case of a square root of $t_{a_1}$, we get a contradiction.		
	\end{proof}
	In Section \ref{sec:geom:express} we will use several relations that hold in the mapping class groups to construct examples of roots of Dehn twists. 
	\begin{prop}[Chain relation {\cite[Proposition 4.12]{MargaliFarb}}]\label{chain}
		Let $a_1, \ldots, a_k$, $k\geq~1$, be a chain of circles in an orientable surface $S$. If $k$ is odd, then the boundary of a closed regular neighbourhood of $\bigcup\limits_{i=1}^{k}a_i$
		consists of two circles $c_1$ and $c_2$ and 
		\[(t_{a_1}^2t_{a_2}\cdots t_{a_k})^k = t_{c_1}t_{c_2}. \]
		If $k$ is even, the boundary of a closed regular neighbourhood of $\bigcup\limits_{i=1}^{k}a_i$
		consists of one circle $c$ and
		\[\left((t_{a_1}\cdots t_{a_k})^2\right)^{k+1} = t_c. \]
	\end{prop}
	\begin{prop}[Star relation \cite{Gervais_top}]\label{star}
		Let $c_1, c_2, c_3$ denote the three boundary components of a surface $S_{1,3}$. Let $a_0, a_1,a_2,a_3$ be circles as in Figure \ref{fig:star}.
		\begin{figure}[h]
			\begin{center}
				\includegraphics[width=0.29\textwidth]{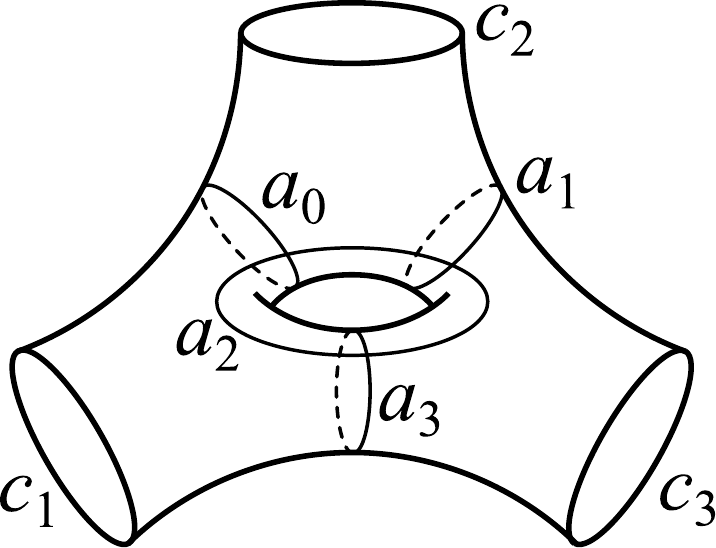}
				\caption{Circles defining the star relation.}\label{fig:star} 
			\end{center}
		\end{figure}
		Then
		\[(t_{a_0}t_{a_1}t_{a_2}t_{a_3})^3 = t_{c_1}t_{c_2}t_{c_3}.\]
	\end{prop}
	The following relation is a common generalization of both the chain and the star relations.
	\begin{prop}[Trident relation]\label{tw:trident}
		Let $c_1,c_2,c_3,a_0,a_1,\ldots,a_{2g+1}$ be circles in an oriented surface $S_{g,3}$ of genus  $g\geq 1$ with three boundary components as in Figure \ref{fig_trident}. 
		\begin{figure}[h]
			\begin{center}
				\includegraphics[width=0.7\textwidth]{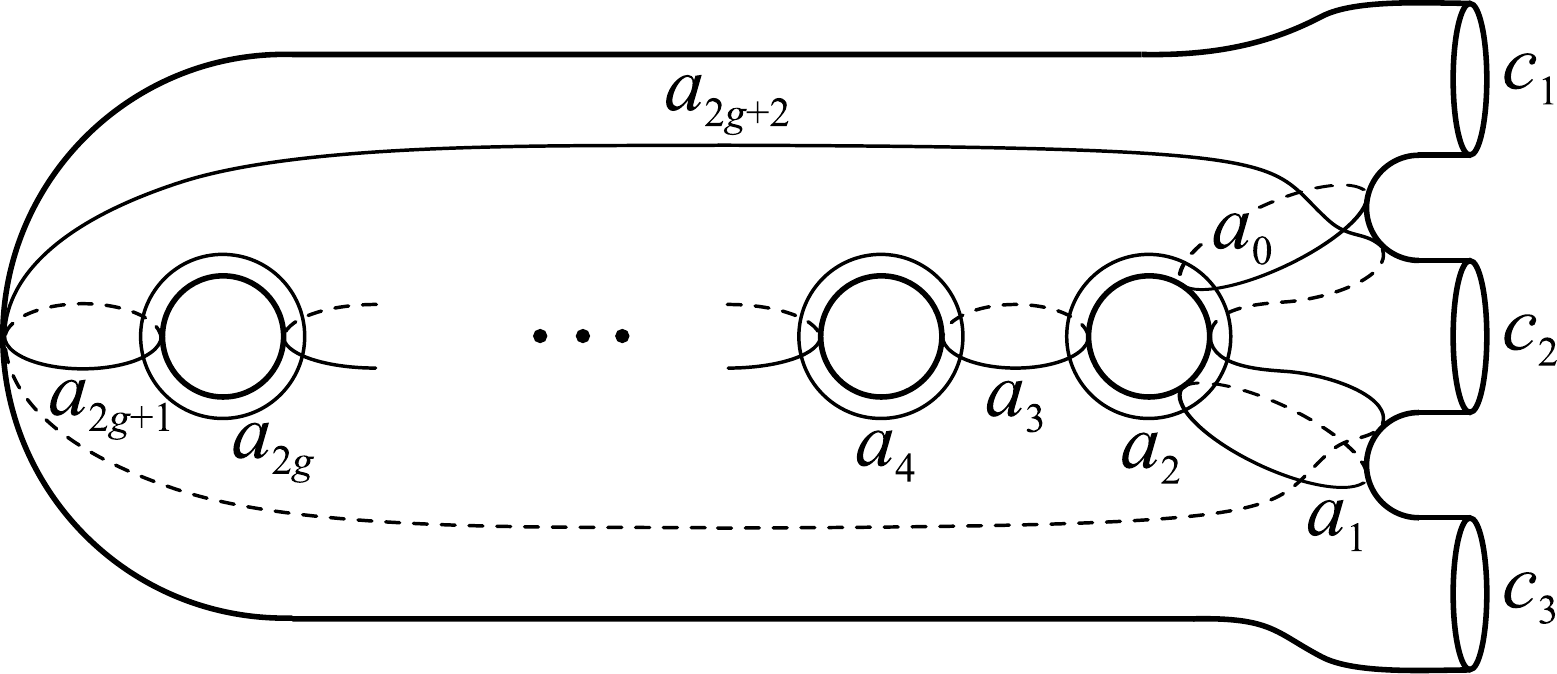}
				\caption{Circles defining the trident relation.}\label{fig_trident} 
			\end{center}
		\end{figure}
		Then
		\[(t_{a_0}t_{a_1}t_{a_2}\cdots t_{a_{2g+1}})^{2g+1}=t_{c_1}t_{c_2}^gt_{c_3}.\]
	\end{prop}
	\begin{proof}
		Let $T=t_{a_0}t_{a_1}t_{a_2}\cdots t_{a_{2g+1}}$ and let $a_{2g+2}$ be as in Figure \ref{fig_trident}. 
		It is straightforward to check that $T(a_k)=a_{k+1}$ for $k=2,3,\ldots, 2g+1$, and $T(a_{2g+2})=a_2$. In particular,
		\begin{equation}\label{eq:trident:a}
			T^{2g+1}(a_k)=a_k=t_{c_1}t_{c_2}^gt_{c_3}(a_k),\text{ for $k=2,3,\ldots,2g+1$}.\end{equation}
		Now let $\lambda_1,\lambda_2,\ldots,\lambda_{2g+2}$ and $\mu_1,\mu_2,\ldots,\mu_{2g+1}$ be arcs as in Figure \ref{fig_trident2}.
		\begin{figure}[h]
			\begin{center}
				\includegraphics[width=1\textwidth]{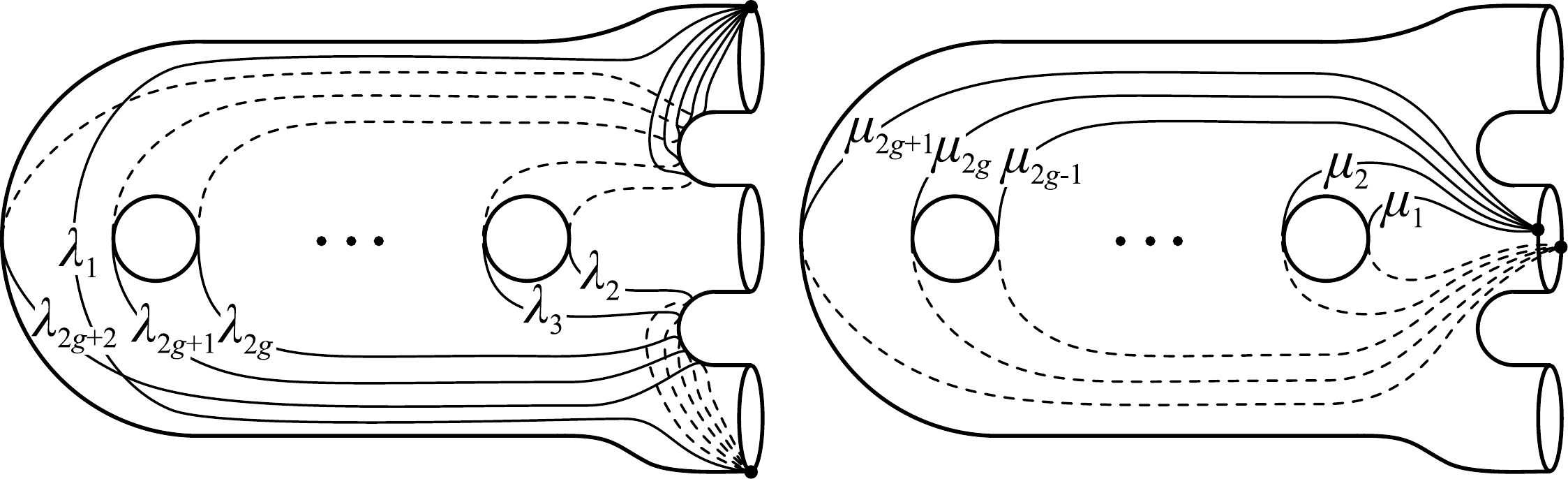}
				\caption{Arcs $\lambda_i$ and $\mu_i$ -- the proof of the trident relation.}\label{fig_trident2} 
			\end{center}
		\end{figure}   
		Then $T(\lambda_k)=\lambda_{k+1}$ for $k=1,2,\ldots,2g+1$. In particular,
		\begin{equation}\label{eq:trident:lambda}
			T^{2g+1}(\lambda_1)=\lambda_{2g+2}=t_{c_1}t_{c_2}^gt_{c_3}(\lambda_1).
		\end{equation}
		As for the circles $\mu_k$, if $H$ is a half twist about $c_2$ such that $H^2=t_{c_2}$, then $H^{-1}T(\mu_k)=\mu_{k+1}$, for $k=1,2,\ldots,2g$, and $T(\mu_{2g+1})=\mu_1$. Therefore,
		\begin{equation}\label{eq:trident:mu}\begin{aligned}
				&T(H^{-1}T)^{2g}(\mu_1)=\mu_1,\\
				&T^{2g+1}(\mu_1)=H^{2g}(\mu_1)=t_{c_2}^g(\mu_1)=t_{c_1}t_{c_2}^gt_{c_3}(\mu_1).
			\end{aligned}\end{equation}
			Equations \eqref{eq:trident:a}, \eqref{eq:trident:lambda} and \eqref{eq:trident:mu} imply that $T^{2g+1}$ and $t_{c_1}t_{c_2}^gt_{c_3}$ agree on all circles and arcs in 
			\[\left\{a_2,a_3,\ldots,a_{2g+1},\lambda_1,\mu_1\right\},\]
			and this set satisfies the assumptions of the Alexander Method (Proposition 2.8 of \cite{MargaliFarb}). Hence, $T^{2g+1}=t_{c_1}t_{c_2}^gt_{c_3}$.
		\end{proof}
		\begin{uw}
		As B\l a\.zej Szepietowski pointed out to us, Proposition \ref{tw:trident} can be also deduced from Proposition 2.12 of \cite{ParLab}.
		\end{uw}
		\begin{wn}\label{wn:trident}
			Let $c_1,c_2,c_3,a_0,a_1,\ldots,a_{2g+1}$ be circles in an oriented surface $S_{g,3}$ of genus  $g\geq 1$ with three boundary components as in Figure \ref{fig_trident}. If we glue a M\"obius strip along $c_2$, obtaining a nonorientable surface $N_{2g+1,2}$ of genus $2g+1$ with two boundary components, then
			\[(t_{a_0}t_{a_1}t_{a_2}\cdots t_{a_{2g+1}})^{2g+1}=t_{c_1}t_{c_3}.\]
		\end{wn}
		\begin{proof}
		 Recall (Theorem 3.4 of \cite{Epstein}) that the Dehn twist about a circle that bounds a  M\"obius strip is trivial.
		\end{proof}
		\begin{wn}\label{wn:trid:2}
		Let $a_i,b_i,c_1,c_2$, $i=0,1,\ldots,2g+1$, be circles in an oriented surface $S_{2g+1,2}$ of genus $2g+1\geq 3$ with two boundary components as in Figure \ref{fig_wn_tr2}. 
				\begin{figure}[h]
			\begin{center}
				\includegraphics[width=1\textwidth]{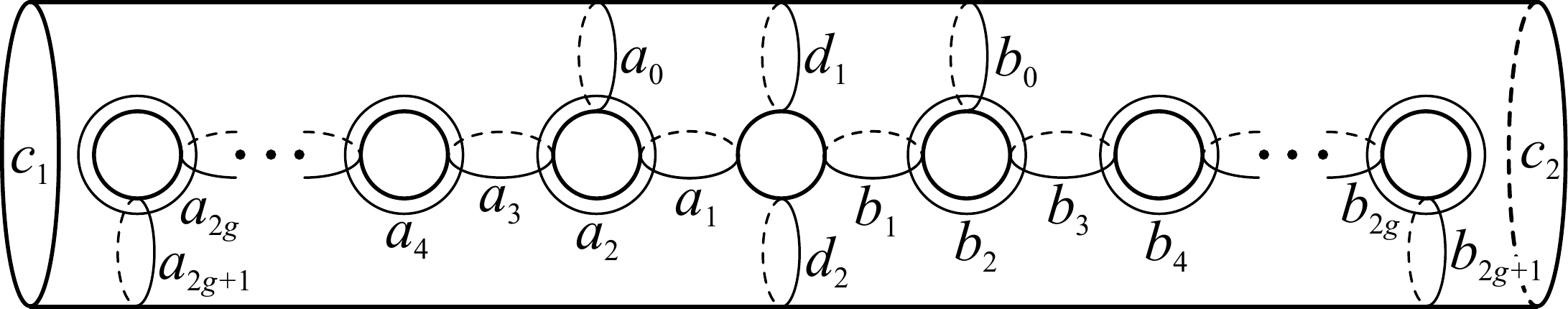}
				\caption{Circles $a_i,b_i,c_1,c_2,d_1,d_2$ -- Corollary \ref{wn:trid:2}.}\label{fig_wn_tr2} 
			\end{center}
		\end{figure}   
		Then
		\[\left(t_{a_0}t_{a_1}t_{a_2}\cdots t_{a_{2g+1}}\left(t_{b_0}t_{b_1}t_{b_2}\cdots t_{b_{2g+1}}\right)^{-1}\right)^{2g+1}=t_{c_1}t_{c_2}^{-1}.\]
		\end{wn}
		\begin{proof}
		 By Proposition \ref{tw:trident}, 
		 \[\begin{aligned}
		    &(t_{a_0}t_{a_1}t_{a_2}\cdots t_{a_{2g+1}})^{2g+1}=t_{c_1}t_{d_1}^gt_{d_2}\\
		    &(t_{b_0}t_{b_1}t_{b_2}\cdots t_{b_{2g+1}})^{2g+1}=t_{c_2}t_{d_1}^gt_{d_2},
		   \end{aligned}\]
		   where $d_1$ and $d_2$ are circles depicted in Figure \ref{fig_wn_tr2}. In order to finish the proof, observe that $t_{a_0}t_{a_1}\cdots t_{a_{2g+1}}$ and $t_{b_0}t_{b_1}\cdots t_{b_{2g+1}}$ have disjoint supports, hence they commute.
		\end{proof}
		\begin{wn}\label{wn:trid:3}
		 Let $a_i,b_i,c_1,c_2$, $i=0,1,\ldots,2g+1$, be circles in an oriented surface $S_{2g,2}$ of genus $2g\geq 2$ with two boundary components as in Figure \ref{fig_wn_tr3}. 
				\begin{figure}[h]
			\begin{center}
				\includegraphics[width=1\textwidth]{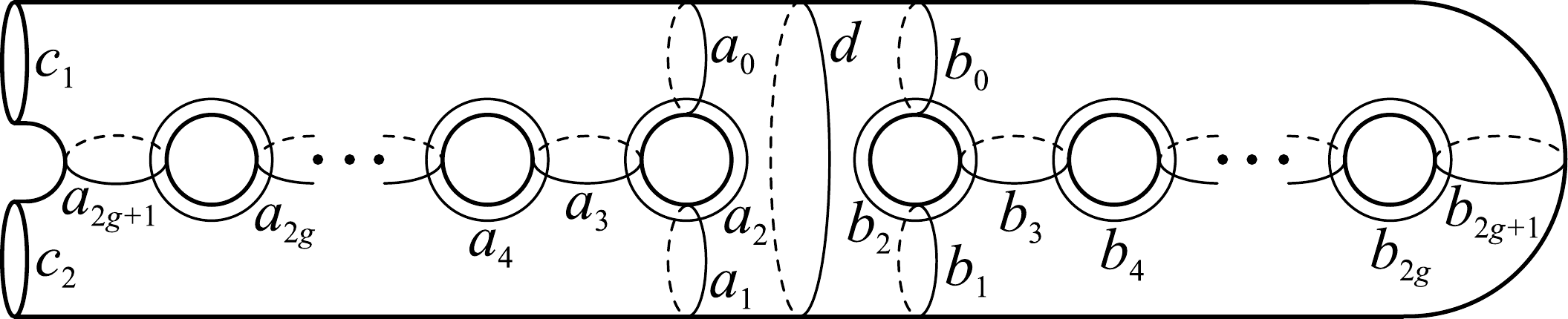}
				\caption{Circles $a_i,b_i,c_1,c_2,d$ -- Corollary \ref{wn:trid:3}.}\label{fig_wn_tr3} 
			\end{center}
		\end{figure}   
		Then
		\[\left(t_{a_0}t_{a_1}t_{a_2}\cdots t_{a_{2g+1}}\left(t_{b_0}t_{b_1}t_{b_2}\cdots t_{b_{2g+1}}\right)^{-1}\right)^{2g+1}=t_{c_1}t_{c_2}.\]
		\end{wn}
		\begin{proof}
		 By Proposition \ref{tw:trident}, 
		 \[\begin{aligned}
		    &(t_{a_0}t_{a_1}t_{a_2}\cdots t_{a_{2g+1}})^{2g+1}=t_{c_1}t_{d}^gt_{c_2}\\
		    &(t_{b_0}t_{b_1}t_{b_2}\cdots t_{b_{2g+1}})^{2g+1}=t_{d}^g,
		   \end{aligned}\]
		   where $d$ is the circle depicted in Figure \ref{fig_wn_tr3}. In order to finish the proof, observe that $t_{a_0}t_{a_1}\cdots t_{a_{2g+1}}$ and $t_{b_0}t_{b_1}\cdots t_{b_{2g+1}}$ have disjoint supports, hence they commute.
		\end{proof}
 				\FloatBarrier
\section{Data sets of roots of Dehn twists}\label{sec:data}
Following Section 2 of \cite{McCullough2010} we define a \emph{data set} of degree $n$ as a tuple
\[(n,g_0,(a,b);(c_1,n_1),\ldots,(c_m,n_m)),\]
where $n,g_0,n_i$ are non-negative integers, $a$, $b$ are residue classes modulo $n$, and each $c_i$ is a residue class modulo $n_i$. Moreover, we assume that:
\begin{itemize}
 \item[(D1)] $n>1$ is odd, each $n_i>1$, and each $n_i$ divides $n$,
 \item[(D2)] $\gcd(a,n)=\gcd(b,n)=1$ and each $\gcd(c_i,n_i)=1$.
\end{itemize}
If a data set satisfies additionally the following condition 
\begin{itemize}
 \item[(D3A)] $g_0\geq 1$ and $b\pm a=ab\mod n$,
\end{itemize}
we say that this data set is of \emph{type \TypeA}. If a data set satisfies (D1), (D2) and 
\begin{itemize}
 \item[(D3B)] $g_0\geq 0$ and $b-a=ab\mod n$,
 \item[(D4B)] $a+b+\sum_{i=1}^m \frac{n}{n_i}c_i=0\mod n$,
\end{itemize}
then we say that this data set is of \emph{type \TypeB}.

We define the \emph{genus} of a data set by the formula:
\[g=\begin{cases}
g_0n+\sum\limits_{i=1}^m\frac{n}{n_i}(n_i-1)&\text{if a data set is of type \TypeA,}\\
2g_0n+\sum\limits_{i=1}^m\frac{n}{n_i}(n_i-1)&\text{if a data set is of type \TypeB.}
    \end{cases}\]
Two data sets of type \TypeB\ are \emph{equivalent} if they differ by an involution 
\[(a,b,c_1,c_2,\ldots,c_m)\mapsto (-b,-a,-c_1,-c_2,\ldots,-c_m)\]
or reordering the pairs $(c_1,n_1),\ldots,(c_m,n_m)$. Two data sets of type \TypeA\ are \emph{equivalent} if they differ by a finite number of the following changes:
\begin{enumerate}
 \item interchanging $a$ and $b$ (note that this is possible only if both data sets satisfy $b+a=ab\mod n$);
 \item changing $(a,b)$ into $(-b,-a)$ (note that this is possible only if both data sets satisfy $b-a=ab\mod n$);
 \item changing the sign of $b$ or any of $c_1,\ldots,c_m$;
 \item reordering the pairs $(c_1,n_1),\ldots,(c_m,n_m)$.
\end{enumerate}
\begin{tw}\label{tw:orb}
Let $c$ be a nonseparating two-sided circle in a nonorientable surface $N_{g+2}$ of genus $g+2\geq 3$, and let $t_{c}$ be the Dehn twist about $c$ with respect to some fixed orientation of the regular neighbourhood of $c$.  For the given odd integer $n>1$ the conjugacy classes of roots of $t_{c}$ of degree $n$ correspond to the equivalence classes of data sets of genus $g$, degree $n$, and type
\[ \begin{cases} \text{\TypeA}& \text{if $N_g\bez c$ is nonorientable,} \\ \text{\TypeB}&\text{if $N_g\bez c$ is orientable.}\end{cases} \]
\end{tw}
\begin{proof}
 The general line of the proof is exactly the same as the proof of Theorem 2.1 in \cite{McCullough2010}, but there are some essential differences, hence we sketch the argument indicating the changes forced by the nonorientability of $N_{g+2}$.
 
 We first construct a data set for a given root $h$ of degree $n$ of $t_{c}$. As was observed by Margalit and Schleimer \cite{MargSchleim}, if $h^n=t_c$, then $h$ must preserve the isotopy class of $c$. Moreover, $h$ must preserve the chosen orientation of the regular neighbourhood of $c$ and, by Proposition \ref{Prop:odd:degree}, $h$ cannot interchange the sides of $c$.

 As in \cite{McCullough2010}, we cut $N_{g+2}$ along $c$ and fill the obtained two boundary components with disks. As a result, we obtain a closed surface $F$ which is either a nonorientable surface of genus $g$ or an orientable surface of genus $\frac{g}{2}$. We will refer to the first situation as \emph{case \TypeA} and to the second one as \emph{case \TypeB}. 
 
 Up to isotopy we can assume that $h(c)=c$, hence $h$ induces an element $t\in \mathcal{M}(F)$. Moreover, since $h^n=t_c$, $t$ has order $n$ in $\mathcal{M}(F)$. By Kerckhoff solution to the Nielsen realization problem \cite{Kerk2}, we can assume that $t$ is in fact a homeomorphism of $F$ of order $n$. By construction, $t$ has at least two fixed points $P$ and $Q$ which correspond to the center points of the two disks which we glued to $N_{g+2}\bez c$. Therefore, the action of the cyclic group $\gen{t\st t^n}\cong C_n$ on $F$ leads to the orbifold $\mathcal{O}=F/\gen{t}$ of genus $g_0$ with two cone points of order $n$ and possibly other cone points $x_i$, $i = 1, \ldots, m$, of orders $n_i$, where $n_i$ is a divisor of $n$. Moreover, the orbifold fundamental group $\pi^{orb}_1(\mathcal{O})$ of $\mathcal{O}$ fits into the following short exact sequence:
  \[\xymatrix@C=1.5pc{1\ar[r]&\pi_1(F)\ar[r]&\pi^{orb}_1(\mathcal{O})\ar[r]^-{\ro}&C_n\ar[r]&1}.\]
 
 Suppose first that we are in case \TypeB, that is $F$ is orientable. We choose the orientation of $F$ so that it agrees in a neighbourhood of $P$ with the orientation induced by the orientation in the regular neighbourhood of $c$ in $N_{g+2}$. Thus, the orientation of $F$ around $Q$ is opposite to the orientation induced by the orientation in the regular neighbourhood of $c$. Furthermore, $t$ is an orientation preserving homeomorphism of $F$. In particular, if $\frac{2\pi k}{n}$ and $\frac{2\pi l}{n}$ are the rotation angles of $t$ around $P$ and $Q$ respectively (with respect to the chosen orientation of $F$), then the condition $h^n=t_c$ implies that
 \begin{equation}k-l=1\mod n.\label{eq:kl:b}\end{equation}
 
 We now choose generators for the orbifold fundamental group of $\mathcal{O}$ so that 
\[\begin{aligned}
  \pi^{orb}_1(\mathcal{O})=\biggen{&\alpha,\beta,\gamma_1,\ldots,\gamma_m,a_1,b_1,\ldots,a_{g_0},b_{g_0}\bigst \\
 &\alpha^n,\beta^n,\gamma_1^{n_1},\ldots,\gamma_m^{n_m},\alpha\beta\gamma_1\cdots\gamma_m\prod_{j=1}^{g_0}[a_j,b_j]}.
 \end{aligned}\]
 We can assume that lifts of $\alpha^n$ and $\beta^n$ to $\pi_1(F)$ are loops around $P$ and $Q$ respectively, and the lift of $\alpha^n$ is positively oriented. 
 
 We now consider the images of generators of $\pi^{orb}_1(\mathcal{O})$ under the projection $\map{\ro}{\pi^{orb}_1(\mathcal{O})}\gen{t\st t^n}$. The lift of $\alpha$ is a $\frac{1}{n}$-th of a full loop encircling $P$, so $\ro(\alpha)=t^a$, where $ak=1\mod n$ (recall that we assumed that the rotation angle of $t$ at $P$ is $\frac{2k\pi}{n}$) and, similarly, $\ro(\beta)=t^b$ where $bl=1\mod n$. Analogously, if we presume that the rotation angle at each point in the preimage of $x_i$ equals $\frac{2\pi y_i}{n_i}$ with $\gcd(y_i, n_i) = 1$, then $\ro(\gamma_i)=\left(t^{n/n_i}\right)^{c_i}$ where~$c_i$ is a residue class mod $n_i$ satisfying $c_iy_i = 1 \mod n_i$. In particular, we get condition (D2) of the data set and \eqref{eq:kl:b} gives condition \CondB. The long relation in the presentation of $\pi^{orb}_1(\mathcal{O})$ yields condition (D4B) of the data set.
  
 Observe also that $F$, as an orientable surface of genus $\frac{g}{2}$, has Euler characteristic equal to $2-g$. Therefore, by the multiplicativity of the orbifold Euler characteristic, we have
 \[\frac{2-g}{n}=2-2g_0+2\left(\frac{1}{n}-1\right)+\sum_{i=1}^m\left(\frac{1}{n_i}-1\right).\]
 This implies that obtained data set has in fact genus $g$.
 
 Finally, assume that $f\in\mathcal{M}(N_{g+2})$ is such that $h'=fhf^{-1}$ is also a root of $t_c$. If $f$ does not interchange the sides of $c$, then the data sets for $h$ and $h'$ can only differ by a reordering of the pairs $(c_1,n_1),\ldots,(c_m,n_m)$. And if $f$ does interchange the sides of $c$, then $f$ induces an orientation reversing map of $F$ which interchanges $P$ and $Q$. As a result,  homeomorphism $\map{t'}{F}{F}$ induced by $h' = fhf^{-1}$ yields a data set with $(a,b)$ changed into $(-b,-a)$ and $c_i$ changed into $-c_i$. Hence, conjugate roots lead to equivalent data sets.
 
 Now let us turn to case \TypeA, so assume that $F$ is nonorientable. The general line of arguments is similar to that in case \TypeB, but there are some essential differences. 
 
 Observe first that Proposition \ref{Prop:odd:degree} implies that homeomorphism $\map{t}{F}{F}$ is of odd degree, hence it can have only isolated fixed points. Therefore, there are no reflections in the orbifold fundamental group of $\mathcal{O}=F/\gen{t}$. So we can choose generators for this group such that
 \[\begin{aligned}
  \pi^{orb}_1(\mathcal{O})=\biggen{&\alpha,\beta,\gamma_1,\ldots,\gamma_m,\mu_1,\mu_2,\ldots,\mu_{g_0}\bigst \\
 &\alpha^n,\beta^n,\gamma_1^{n_1},\ldots,\gamma_m^{n_m},\alpha\beta\gamma_1\cdots\gamma_m\prod_{j=1}^{g_0}\mu_j^2}.
 \end{aligned}\]
 Moreover, we can assume that lifts of $\alpha^n$ and $\beta^n$ to $\pi_1(F)$ are loops around $P$ and $Q$ respectively, and the lift of $\alpha^n$ to $\pi_1(F)$ is positively oriented with respect to the orientation inherited from the regular neighbourhood of $c$ in~$N_{g+2}$. 
 
 We now have two possibilities: either the lift of $\beta^n$ to $\pi_1(F)$ is positively oriented with respect to the orientation induced from the regular neighbourhood of $c$, or not. 
 In the first case, the rotation angles $\frac{2\pi k}{n}$ and $\frac{2\pi l}{n}$ must satisfy
 \[k+l=1\mod n\]
 and in the second case, they satisfy
 \[k-l=1\mod n.\]
 This leads to condition (D3A).
 As for condition (D2), we obtain it in exactly the same way as in case \TypeB.
 
 Since $F$ is a nonorientable surface of genus $g$, by the multiplicativity of the orbifold Euler characteristic, we get
 \[\frac{2-g}{n}=2-g_0+2\left(\frac{1}{n}-1\right)+\sum_{i=1}^m\left(\frac{1}{n_i}-1\right),\]
 which implies that the obtained data set has genus $g$.
 
 As for conjugacy classes of roots, if $f\in {\mathcal{M}(N_{g+2})}$ is such that $h'=fhf^{-1}$ is also a root of $t_c$, then the map induced by $f$ on $F$ either does not change the rotation angles around $P$ and $Q$ or it interchanges them. If $f$ does not interchange $P$ and $Q$, then it must also preserve the local orientation around $P$ and consequently $f$ maps $(a,b)$ in the first data set to  $(a,b)$ in the second data set. If $f$ interchanges $P$ with $Q$, then it sends $(a,b)$ either to $(b,a)$ (this happens when $b+a=ab\mod n$), or to $(-b,-a)$ (this happens when $b-a=ab\mod n$). 
 Hence $h$ and $h'$ have equivalent data sets.
 
 
Suppose now that we start from a data set. We construct the corresponding orbifold $\mathcal{O}$ (nonorientable in case \TypeA, and orientable in case \TypeB), and the representation
\[\map{\ro}{\pi_1^{orb}(\mathcal{O})}{\gen{t\st t^n}}.\]
In case B it turns out that we can always choose such a generating set for $\pi^{orb}_1(\mathcal{O})$ that $\ro(a_i)=\ro(b_i)=1$. The idea is to start with any generating set and then slide the cone point corresponding to $P$ around loops homologous to $a_i$ or $b_i$ -- for details see \cite{McCullough2010}.

In case A, the situation is slightly more complicated. If we slide the cone point corresponding to $Q$ along a loop homologous to $\mu_1\mu_{2}$, then we get new generators $\alpha',\beta',\gamma_1',\ldots,\gamma_m',\mu_1',\mu_2',\ldots,\mu_{g_0}'$ of $\pi^{orb}_1(\mathcal{O})$ such that $\ro(\mu_1')=\ro(\mu_1)\ro(\beta)$ and $\ro(\mu_{2}')=\ro(\mu_{2})\ro(\beta)$. The images under $\ro$ of the remaining generators do not change. Since $\ro(\beta)$ generates $\gen{t\st t^n}$, we can repeat such slides until for the new $\mu_1$ we obtain $\ro(\mu_1)=1$. By an obvious induction, we can assume that 
\[\ro(\mu_1)=\ro(\mu_2)=\cdots=\ro(\mu_{g_0-1})=1.\]
Note also, that the long relation in $\pi^{orb}_1(\mathcal{O})$ gives
\begin{equation}\ro(\mu_{g_0})^{-2}=\ro(\alpha\beta\gamma_1\cdots\gamma_m).\label{eq:LongRelAro}\end{equation}
Therefore $\ro(\mu_{g_0})$ is completely determined by the values of $\ro$ on the remaining generators of $\pi^{orb}_1(\mathcal{O})$.

As is explained in \cite{McCullough2010}, the kernel of the homomorphism $\ro$ is torsion-free, hence we obtain a covering $F\to \mathcal{O}$, where $F$ is a nonorientable surface of genus $g$ in case \TypeA\ (here we use the fact that $n$ is odd), and $F$ is an orientable surface of genus $\frac{g}{2}$ in case \TypeB. Then we remove disks around the fixed points $P$ and $Q$ corresponding to $\alpha$ and $\beta$, and identify the resulting boundary components with appropriate choice of orientations (depending on the type of the data set). This completes the construction of a root of $t_c$ from a given data set.

We now claim that if we start from two roots of $t_c$ with equivalent data sets 
\[\begin{aligned}
 &(n,g_0,(a,b);(c_1,n_1),\ldots,(c_m,n_m))  \\
 &(n,g_0,(a',b');(c_1',n_1),\ldots,(c_m',n_m))
  \end{aligned}
\]
of the same type, then after an appropriate choice of generators for the respective orbifold fundamental groups and using the natural isomorphism $\gen{t \st t^n} \cong \gen{t' \st (t')^n}$, we can assume that the corresponding representations $\map{\ro}{\pi^{orb}_1(\mathcal{O})}{C_n}$ and
$\map{\ro'}{\pi^{orb}_1(\mathcal{O}')}{C_n}$ are exactly the same, that is $\ro(\alpha)=\ro'(\alpha')$, $\ro(\beta)=\ro'(\beta')$, $\ro(\gamma_i)=\ro'(\gamma_i')$ and additionally  
\[\begin{cases}
   \ro(\mu_i)=\ro'(\mu_i')&\text{ in case A}\\
   \ro(a_i)=\ro'(a_i'),\ \ro(b_i)=\ro'(b_i') &\text{ in case B.}
  \end{cases}
\]
Suppose first that we are in case B. As we observed before, we can assume that both $\ro$ and $\ro'$ map the hyperbolic generators to 1. Now the claim is obvious if $\ro(\alpha)=t^a=t^{a'}=\ro'(\alpha')$ (because then automatically $b=b'$ and $\{c_1,\ldots,c_m\}=\{c_1',\ldots,c_m'\}$), and if $\ro(\alpha)=t^a=t^{-b}=\ro'(\beta')^{-1}$, then we can interchange the roles of $\alpha'$ and $\beta'$  and additionally change the orientation of the orbifold $\mathcal{O}'$. 

In case A, as we argued before, we can assume that $\ro(\mu_i)=\ro'(\mu_i')=1$ for $i= 1,2, \ldots, g_0-1$. If $(a,b)$ and $(a',b')$ satisfy condition (D3A) with a different sign, then we homeomorphically modify the generators of $\pi^{orb}_1(\mathcal{O}')$ by sliding the cone point corresponding to $Q$ around a loop homologous to $\mu_{g_0}$. As a result of such a conversion the new generator $\beta''$ satisfies $\ro'(\beta'')=\ro'(\beta')^{-1}$, the value of $\ro'(\mu_{g_0}')$ changes to $\ro'(\mu_{g_0}'')=\ro'(\mu_{g_0}')\ro'(\beta')$, and the values of $\ro'$ on all other generators of $\pi^{orb}_1(\mathcal{O}')$ remain unchanged. Thus, we can assume that both data sets satisfy condition (D3A) with the same sign. Then, similarly as in case B, we can assume that $\ro(\alpha)=\ro'(\alpha')$ and $\ro(\beta)=\ro'(\beta')$ (we possibly interchange $\alpha'$ with $\beta'$). Moreover, by sliding cone points $x_i$ around a loop homologous to $\mu_{g_0}$, we can assume that $\ro(\gamma_i)=\ro'(\gamma_i')$. Finally, by \eqref{eq:LongRelAro}, $\ro(\mu_{g_0})=\ro'(\mu_{g_0}')$. 

Hence, we can assume that the equivalent data sets correspond to identical representations $\ro$ and $\ro'$ and, as is explained in details in \cite{McCullough2010}, such a situation leads to conjugate roots of $t_c$. 
\end{proof}
As an immediate consequence of the above theorem, we obtain some restrictions on the existence of roots of Dehn twists about nonseparating circles.
\begin{tw} \label{tw:root:existence}
Let $c$ be a nonseparating two-sided circle in a nonorientable surface $N_{g+2}$ of genus $g+2$.
\begin{enumerate}
           \item If $N_{g+2}\bez c$ is nonorientable, then the Dehn twist $t_c$ has a nontrivial root in $\mathcal{M}(N_{g+2})$ if and only if $g=3$ or $g\geq 5$.
           \item If $N_{g+2}\bez c$ is orientable, then $g=2g'$ for some positive integer $g'$ and $t_c$ has a nontrivial root in $\mathcal{M}(N_{g+2})$ if and only if $g'\geq 2$.           
          \end{enumerate}
\end{tw}
\begin{proof} The existence part of the theorem will be proved algebraically in Section \ref{sec:degrees} (Theorems \ref{tw:allrots:a} and \ref{tw:allrots:b})  and geometrically in Section \ref{sec:geom:express}, hence we now concentrate on the 'only if' part. 
 \begin{enumerate}
  \item By conditions (D1) and (D3A), the genus of a data set in case~\TypeA\ satisfies
  \[g=n\left(g_0+\sum_{i=1}^m\left(1-\frac{1}{n_i}\right)\right)\geq 3.\]
  Moreover, if $g=4$, then $g_0=1,n=3$ and 
  $\sum\limits_{i=1}^m\left(1-\frac{1}{n_i}\right)=\frac{1}{3}$,
  but this is not possible.
  \item The statement is trivial for $g=0$, so it is enough to prove that there is no data set of type \TypeB\ if $g=2$.
  By condition (D1), the genus of a data set in case \TypeB\ satisfies
  \[g=n\left(2g_0+\sum_{i=1}^m\left(1-\frac{1}{n_i}\right)\right).\]
  Hence if $g=2$, then $n=3, g_0=0, m=1$ and $n_1=3$. But then condition \CondB\ implies that $a=2$ and $b=1$, which contradicts (D2) and (D4B).
 \end{enumerate}
\end{proof}
%
%
\section{Degrees of roots of Dehn twists.}\label{sec:degrees}
Margalit and Schleimer \cite{MargSchleim} constructed a root of degree $2g+1$ of a Dehn twist $t_{c}$ about a nonseparating circle $c$ in an orientable surface $S_{g+1}$, where $g \geq 1$. Later McCullough, Rajeevsarathy and Monden proved that this is the maximal possible degree of a root of $t_{c}$ (Corollary 2.2 of \cite{McCullough2010} and Corollary~C of \cite{Monden-roots}).

In the case of nonorientable surfaces the problem of determining the maximal degree of a root of a Dehn twist about a nonseparating circle is much more complicated and we give only some partial results.
\begin{tw} \label{tw:all:degs:a}
Let $c$ be a nonseparating two-sided circle in a nonorientable surface $N_{g+2}$ of genus $g+2$, where $g=3$ or $g\geq 5$, such that $N_{g+2}\bez c$ is nonorientable. Let $N$ be the maximal possible degree of a nontrivial root of the Dehn twist $t_c$.
\begin{enumerate}
           \item If $g$ is odd, then $N=g$.
           \item If $g=2^l k$ for some odd integers $l\geq 1$ and $k$, then $N=\frac{2^l+1}{3}\cdot k$. In particular, $\frac{g}{2}\geq N>\frac{g}{3}$           
            \item Let $g=2^lk_1k_2$ for even integer $l\geq 2$ and some positive odd integers $k_1,k_2$ such that $2^lk_1+1=0\mod 3$, and $k_1$ is as small as possible (for a fixed $g$). Then $N=\frac{2^lk_1+1}{3}\cdot k_2$. In particular, $\frac{g}{2}>N>\frac{g}{3}$.
            \item Suppose that $g$ does not satisfy (1) -- (3) above and assume that there exist odd integers $1\leq k\leq\frac{g}{3}$ and $k_1,k_2\geq 1$ such that $g+k=k_1k_2$, where $k_1+1=0\mod 4$ and $k|\frac{k_1+1}{4}\cdot k_2$. Then $\frac{g}{3}>N\geq \frac{k_1+1}{4}\cdot k_2$. In particular, $\frac{g}{3}>N>\frac{g}{4}$.
            \item Suppose that $g$ does not satisfy (1) -- (4) above and assume that $g=2^2k$ for some odd integer $k$. Then $N=k=\frac{g}{4}$.
            \item If $g$ does not satisfy (1) -- (5) above, then $\frac{g}{4}>N>\frac{g}{6}$.
          \end{enumerate}
\end{tw} 
\begin{proof}
By conditions (D1) and (D3A), the genus of a data set in case \TypeA\ satisfies
  \begin{equation}g=n\left(g_0+\sum_{i=1}^m\left(1-\frac{1}{n_i}\right)\right)\geq ng_0\geq n,\label{eq:tw:max:a}\end{equation}
  hence $N\leq g$. 
 \begin{enumerate}
  \item If $g$ is odd, then there is a data set of type A with $n=g$, $g_0=1$, $(a,b)=(2,2)$ and $m=0$. 
  \item Suppose first that $l=1$, that is $g=2k$ and $k$ is odd. In this case there is a data set of type A with $n=k$, $g_0=2$, $(a,b)=(2,2)$ and $m=0$.
  
  
  It remains to show that this is the maximal possible degree. Observe that if $g$ is even, then $g_0$ cannot be odd, because otherwise the right hand side of 
  \[g=ng_0+\sum_{i=1}^m\left(n-\frac{n}{n_i}\right)\]
  would be odd. Hence $g_0\geq 2$ and \eqref{eq:tw:max:a} implies that $n\leq \frac{g}{g_0}\leq \frac{g}{2}$. This completes the proof for $l=1$.
  
   Note that $n=\frac{g}{2}$ implies that $m=0$ and $g=2n$ which proves that if $l>1$ or in cases (3) -- (6) we have either $g_0\geq 4$ or $g_0=2$ and $m\geq 1$. Moreover, if $g_0\geq 4$, then \eqref{eq:tw:max:a} implies that $n\leq \frac{g}{g_0}\leq \frac{g}{4}$, therefore assume that $g_0=2$.  If $m\geq 2$, then 
  \[g=2n+\sum_{i=1}^m\left(n-\frac{n}{n_i}\right)\geq 2n+2\left(n-\frac{n}{3}\right)=\frac{10}{3}n>3n,\]
  so assume that $m=1$. In this case, we have 
  \[g=2^lk=3n-\frac{n}{n_1}=\frac{n}{n_1}(3n_1-1),\]
  hence $3n_1-1=2^ls$ for some integer $s$ such that $s\big|k$ and $\frac{n}{n_1}=\frac{k}{s}$. As a result
  \[n=n_1\cdot \frac{n}{n_1}=\frac{2^ls+1}{3}\cdot \frac{k}{s}=\frac{2^l+\frac{1}{s}}{3}\cdot k\leq \frac{2^l+1}{3}\cdot k\]
  with equality if and only if $s=1$. On the other hand, there is a data set of type A with $n=\frac{2^l+1}{3}\cdot k$, $g_0=2$, $m=1$, $n_1=\frac{2^l+1}{3}$ and $(a,b,c_1)=(2,2,1)$. 
   \item As we observed in the proof of the previous point, we can assume that $g=2$, $m=1$ and then 
   \[n=\frac{2^ls+1}{3}\cdot \frac{k_1k_2}{s}=\frac{2^l+\frac{1}{s}}{3}\cdot k_1k_2\]
   for some integer $s$ such that $s\big|k_1k_2$ and $n_1=\frac{2^ls+1}{3}\in \zz$. By assumption, $s=k_1$ is the smallest positive integer with these properties (for a fixed $g$), hence 
   \[n\leq \frac{2^lk_1+1}{3}\cdot \frac{k_1k_2}{k_1}=\frac{2^lk_1+1}{3}\cdot k_2.\]
   This maximal possible value of $n$ is realised by the data set of type~A with $g_0=2$, $m=1$, $n_1=\frac{2^lk_1+1}{3}$ and $(a,b,c_1)=(2,2,1)$.
  \item If $g$ does not satisfy the assumptions of (1) -- (3), then either $g_0\geq 4$ or $g_0=2$ and $m\geq 2$. In both cases $g>3n$. On the other hand, if $g$ satisfies (4), then there is a data set of type A with $n=\frac{k_1+1}{4}\cdot k_2$, $g_0=2$, $m=2$, $n_1=\frac{k_1+1}{4}$, $n_2=\frac{n}{k}$ and $(a,b,c_1,c_2)=(2,2,1,1)$.
  \item If $g=2^2k$ for an odd integer $k$, then there is a data set of type A with $n=k$, $g_0=4$, $m=0$ and $(a,b)=(2,2)$. It remains to show that $g\geq 4n$ if $g$ does not satisfy assumptions of (1) -- (4). 
 
  As we observed in the proof of the previous point, if $g$ does not satisfy the assumptions of (1) -- (3), then either $g_0\geq 4$ or $g_0=2$ and $m\geq 2$. Moreover, $g_0=4$ implies that $n\leq \frac{g}{4}$, hence assume that $g_0=2$. If $m\geq 3$, then 
  \[g=2n+\sum_{i=1}^m\left(n-\frac{n}{n_i}\right)\geq 2n+3\left(n-\frac{n}{3}\right)= 4n,\]
  so assume that $m=2$. In this case, we have
  \[g=2n+\left(n-\frac{n}{n_1}\right)+\left(n-\frac{n}{n_2}\right)=-\frac{n}{n_2}+\frac{n}{n_1}(4n_1-1).\]
  If we set $k_1=4n_1-1$, $k_2=\frac{n}{n_1}$ and $k=\frac{n}{n_2}$, then we obtain positive odd integers such that $k\leq \frac{n}{3}<\frac{g}{3}$ and 
  \[g+k=\frac{n}{n_1}(4n_1-1)=k_1k_2.\]
  This implies that $g$ satisfies (4).
  \item At this point we know that if $g$ does not satisfies (1) -- (4) then $g\geq 4n$. Moreover, if $g=4n$, then $g$ satisfies (5), so $g>4n$ in case~(6).
   Observe also that $g$ is even and $g\neq 2\mod 6$ (because otherwise $g$ would satisfy the assumptions of (1), (2) or (3)). If $g=6k-2$ for some integer $k>0$, then $k$ must be odd, because otherwise $g$ would satisfy the assumptions of (2). Hence, there is a data set of type A with $n=k$, $g_0=4$, $m=2$, $n_1=n_2=n$ and $(a,b,c_1,c_2)=(2,2,1,1)$. In particular, $N\geq \frac{g+2}{6}>\frac{g}{6}$.
  
  Finally, if $g=0\mod 6$, then $g=2^{2l}\cdot 3k$ for some odd integer $k$ and $l\geq 2$ (because otherwise $g$ would satisfy the assumptions of (2) or (5)). In this case there exists a data set of type A with $n=\frac{2^{2l}\cdot 3+6}{6}\cdot k=(2^{2l-1}+1)k$, $g_0=4$, $m=2$, $n_1=n_2=\frac{n}{3k}$ and $(a,b,c_1,c_2)=(2,2,1,1)$. In particular, $N>\frac{g}{6}$.
 \end{enumerate}
\end{proof}
 \begin{uw}
  It can be checked using a computer (with a GAP \cite{GAP4} program similar to that described in Section 3 of \cite{McCullough2010}) that there are 10 values of $g\leq 500$ for which $N<\frac{g}{4}$ (that is, values of $g$ which do not satisfy the assumptions of cases (1) -- (5) of Theorem \ref{tw:all:degs:a}). 
  \begin{center}
   \begin{tabular}{|c|c|c|}
    \hline
    $g$&$g/N$&Data set for a maximal root\\
    \hline\hline
    $16=2^4$&5.33&$(3,4,(2,2);(1,3),(1,3))$\\ \hline
    $48=2^4\cdot 3$&5.33&$(9,4,(2,2);(1,3),(1,3))$\\ \hline
    $64=2^6$&4.27&$(15,2,(2,2);(1,5),(1,5),(1,3))$\\ \hline    
    $112=2^4\cdot 7$&4.87&$(23,2,(2,2);(1,23),(1,23),(1,23))$\\ \hline    
    $144=2^4\cdot 3^2$&4.97&$(29,4,(2,2);(1,29))$\\ \hline
    $192=2^6\cdot 3$&4.27&$(45,2,(2,2);(1,5),(1,5),(1,3))$\\ \hline    
    $256=2^8$&5.69&$(45,4,(2,2);(1,9),(1,5))$\\ \hline
    $304=2^4\cdot 19$&4.83&$(63,2,(2,2);(1,63),(1,63),(1,7))$\\ \hline        
    $336=2^4\cdot 3\cdot 7$&4.87&$(69,2,(2,2);(1,23),(1,23),(1,23))$\\ \hline    
    $496=2^4\cdot 31$&4.72&$(105,2,(2,2);(1,15),(1,15),(1,7))$\\ \hline
   \end{tabular}
  \end{center}
The above examples indicate that there is no much room to improve the general bound $N>\frac{g}{6}$ of Theorem \ref{tw:all:degs:a}.
 \end{uw}
As we will see in the proof of Theorem \ref{tw:max:b} below, case B is much more involving, because in this case there is an additional arithmetic condition (D4B) that imposes further constraints on the existence of data sets. For this reason we need the following two technical lemmas.
\begin{lem}\label{Lem:tw:b:max}
 Let $n=n_1d$, where $n_1,d\geq 3$ are odd integers. Consider the following system of equations
 \begin{equation}
  \begin{cases}\label{sys:eq:abc}
   \gcd(a,n)=\gcd(b,n)=\gcd(c_1,n_1)=1\\
   b-a=ab\mod n\\
   a+b+c_1d=0\mod n
  \end{cases}
 \end{equation}
 in three variables $a,b,c_1\in \zz$.
 \begin{enumerate}
  \item If $n_1=0\mod 3$ and $d\neq 0\mod 3$, then system \eqref{sys:eq:abc} has no solution.
  \item If $n_1=d=0\mod 3$ or $n_1\neq 0\mod 3$, then system \eqref{sys:eq:abc} has a solution.
 \end{enumerate}
\end{lem}
\begin{proof}
 If $(a,b,c_1)$ is a solution of \eqref{sys:eq:abc}, then 
 \[\begin{cases}
 b-a=ab\mod d\\
    a+b=-c_1d=0\mod d.
   \end{cases}
\]
These two conditions imply that $(a,b)=(2,-2)\mod d$, so we can assume that
  \begin{equation*}a=a_1d+2,\ b=b_1d-2,\ \text{for some $a_1,b_1\in\{0,1,\ldots,n_1-1\}$}.\end{equation*}
  System \eqref{sys:eq:abc} yields
  \[\begin{cases}
  0=ab-b+a=(a_1b_1d+b_1-a_1)d\mod dn_1\\
       0=a+b+c_1d=(a_1+b_1+c_1)d\mod dn_1.
    \end{cases}
\]
Equivalently,
\begin{equation}\label{eq:pf:b:mxy:set}
 \begin{cases}
 a_1b_1d+b_1-a_1=0\mod n_1\\
  a_1+b_1+c_1=0\mod n_1. 
 \end{cases}
  \end{equation}
 Note that if $a_1=0$ or $b_1=0$, then the first equation of \eqref{eq:pf:b:mxy:set} implies that $a_1=b_1=0$. But then the second equation gives $c_1=0\mod n$, which is a contradiction with \eqref{sys:eq:abc}. Hence $a_1\neq 0$ and $b_1\neq 0$.
 
 If $n_1=0\mod 3$, then the second equation of \eqref{eq:pf:b:mxy:set} implies that $a_1=b_1\neq 0\mod 3$ (otherwise $c_1=0\mod 3$). But then, the first equation of \eqref{eq:pf:b:mxy:set} implies that $d=0\mod 3$. This completes the proof of (1).
 
 Assume now that $n_1=d=0\mod 3$. Define $p=p_1\cdot p_2\cdots p_i$ to be the product of all common prime divisors of $\frac{d}{3}$ and $n_1$, and let $q=q_1\cdot q_2 \cdots q_j$ be the product of all prime divisors of $n_1$ different from $p_1,\ldots,p_i$. If one of these sets of primes is empty, we define $p=1$ or $q=1$ respectively. Let $a_1$ be a solution of the set of congruences 
\[\begin{cases}
   a_1=-\left(\frac{d}{3}\right)^{-1}\mod q\\
   a_1=1\mod p.
  \end{cases}
\]
(If $p=1$ or $q=1$, then we solve only one equation.) It is apparent from this definition that $\gcd(a_1,n_1)=1$ and
\[\begin{cases}
   a_1d=-3\mod q\\
   a_1d=0\mod p.
  \end{cases}
\]
Hence, $\gcd(a_1d+1,n_1)=\gcd(a_1d+2,n_1)=1$ and we can define $b_1$ as $a_1(a_1d+1)^{-1}\mod n_1$. It is straightforward to check that:
\begin{itemize}
\item If $c_1=-(a_1+b_1) \mod n_1$, then 
\[\begin{aligned}&\gcd(-c_1,n_1)=\gcd(a_1+a_1(a_1d+1)^{-1},n_1)\\
&=\gcd(a_1(a_1d+1)+a_1,n_1)=\gcd(a_1(a_1d+2),n_1)=1.
\end{aligned}\]
\item If $a=a_1d+2$ and $b=b_1d-2$, then $\gcd(a,d)=\gcd(b,d)=\gcd(a,n_1)=1$ and 
 \[(a_1d)(b_1d)+b_1d-a_1d=0\mod n_1.\]
 This implies that 
 \[\begin{cases}
   2b_1d=3\mod q\\
   b_1d=0\mod p.
  \end{cases}\]
Hence, $\gcd(b,n_1)=1$ and $\gcd(a,n)=\gcd(b,n)=1$.
 \item $(a_1,b_1,c_1)$ satisfy \eqref{eq:pf:b:mxy:set}.
\end{itemize}
This completes the proof that 
\[(a,b,c_1)=(a_1d+2,b_1d-2,-(a_1+b_1))\] is a solution of \eqref{sys:eq:abc}. 

Finally, if $n_1\neq 0\mod 3$, then we can simplify the above argument by defining $p=p_1\cdot p_2\cdots p_i$ to be the product of all common prime divisors of $d$ and $n_1$, and $q=q_1\cdot q_2 \cdots q_j$ be the product of all prime divisors of $n_1$ different from $p_1,\ldots,p_i$. We then define $a_1$ as a solution of the set of congruences 
\[\begin{cases}
   a_1=-3d^{-1}\mod q\\
   a_1=1\mod p.
  \end{cases}
\]
As before, this gives 
\[\begin{cases}
   a_1d=-3\mod q\\
   a_1d=0\mod p
  \end{cases}
\]
and we can repeat the same argument as in the case of $n_1=d=0\mod 3$. This completes the proof of (2).
\end{proof}
\begin{lem}\label{Lem:tw:b:max2}
 Let $n\geq 3$ be an odd integer. The following system of equations
 \begin{equation}
  \begin{cases}\label{sys:eq:abc2}
   \gcd(a,n)=\gcd(b,n)=\gcd(c_1,n)=1\\
   b-a=ab\mod n\\
   a+b+c_1=0\mod n
  \end{cases}
 \end{equation}
 in three variables $a,b,c_1\in \zz$ has a solution if and only if $n\neq 0\mod 3$.
\end{lem}
\begin{proof}
 If $(a,b,c_1)$ is a solution of \eqref{sys:eq:abc2}, then clearly $a,b\neq 0\mod n$ and $a\neq b\mod n$. In particular, if $n=0\mod 3$, then $a+b=0\mod 3$ and the third equation of \eqref{sys:eq:abc2} implies that $c_1=0\mod 3$, which is a contradiction.
 
 On the other hand, if $n\neq 0\mod 3$, then 
 \[(a,b,c_1)=\left(\frac{n+3}{2},-3,\frac{n+3}{2}\right)\]
 is a solution of \eqref{sys:eq:abc2}.
\end{proof}
\begin{tw} \label{tw:max:b}
Let $c$ be a nonseparating two-sided circle in a nonorientable surface $N_{2g'+2}$ of genus $2g'+2$, where $g'\geq 2$, such that $N_{2g'+2}\bez c$ is orientable. Let $N$ be the maximal possible degree of a nontrivial root of the Dehn twist $t_c$.
\begin{enumerate}
	   \item If $g'=3k$ for some odd integer $k$, then $N=3g'$.
	   \item If $g'=2k$ for some odd integer $k$, then $N=\frac{5}{2}g'$.
	   \item If $g'=2^2\cdot 3k$ for some odd integer $k$, then $N=\frac{9}{4}g'$.
	   \item If $g'$ does not satisfy (1) and $g'=5k$ for some odd integer $k$, then $N=\frac{11}{5}g'$.
	   \item If $g'=2^3k$ for some odd integer $k$, then $N=\frac{17}{8}g'$.
	   \item If $g'$ does not satisfy (1) nor (4) and $g'=11k$ for some odd integer $k$, then $N=\frac{23}{11}g'$. 
	   \item If $g'=2^4\cdot 3k$ for some odd integer $k$, then $N=\frac{33}{16}g'$.
	   \item If $g'$ does not satisfy (1), (4) nor (6) and $g'=17k$ for some odd integer $k$, then $N=\frac{35}{17}g'$. 	   
	   \item If $g' = l\cdot 3k$, where $l=1\mod 3$, $k$ is odd and $g'$ does not satisfy (1) -- (8), then $\frac{41}{20}g' \geq N \geq\frac{2l+1}{l}\cdot g' > 2g'$. 
	    \item If $g' = l\cdot k$, where $l=2\mod 3$, $k$ is odd and $g'$ does not satisfy (1) -- (9), then $\frac{41}{20}g' \geq N \geq\frac{2l+1}{l}\cdot g' > 2g'$. If additionally $g'$ is prime, then $N=2g'+1$.
           \item If $g'$ does not satisfy (1) -- (10) and $g'=1\mod 6$, then $\frac{5}{4}g'\geq N\geq  g'$. 
           \item If $g'$ does not satisfy (1) -- (11), then $g'=4\mod 12$ and $\frac{5}{4}g'\geq N\geq g'+1$. 
          \end{enumerate}
\end{tw}
\begin{proof}
 The genus of a data set in case \TypeB\ satisfies
  \[2g'=n\left(2g_0+\sum_{i=1}^m\left(1-\frac{1}{n_i}\right)\right).\]
  This implies that if $g_0\geq 1$, then $n\leq g'$. Thus, for the rest of the proof assume that $g_0=0$ and therefore $m \geq 1$. 
  But when $m\geq 3$, then 
   \[2g'\geq 3\left(n-\frac{n}{3}\right)=2n,\]
  hence for the rest of the proof we assume that $m=1$ or $m=2$.
  
  Note also that if $m=2$, then
  \begin{equation}2g'\geq 2\left(n-\frac{n}{3}\right)=\frac{4}{3}n,\label{eq:proof:b:max:23}\end{equation}
  which implies that $n\leq \frac{3}{2}g'$. 
  
  Suppose first that $m=1$. Thus
  \begin{equation}\label{eq:pf:b:mx:ab}g'=\frac{1}{2}n\left(1-\frac{1}{n_1}\right)=\frac{n}{n_1}\cdot \frac{n_1-1}{2}=d\cdot \frac{n_1-1}{2},\end{equation}
  where $d=\frac{n}{n_1}$. In particular, for a fixed $g'$, smaller values of $n_1$ yield bigger values of $n$.  
  Let analyse the existence of data sets satisfying \eqref{eq:pf:b:mx:ab} for small values of $n_1$.


If $n_1=3$, then $d>1$ and Lemma \ref{Lem:tw:b:max} implies that $d=0\mod 3$, therefore $g'=3k$ for some odd integer $k$. The same lemma guarantees the existence of $a,b,c_1$ satisfying (D2), (D3B) and (D4B), hence there is a data set of type B with $n=3g'$. This completes the proof of (1).


If $n_1=5$, then \eqref{eq:pf:b:mx:ab} implies that $g'=2k$ for odd $k$. Conversely, if $g'=2k$ for some odd integer $k$, then by Lemmas \ref{Lem:tw:b:max} and \ref{Lem:tw:b:max2} there is a data set of type B with $n=\frac{5}{2}g'$ (we need Lemma \ref{Lem:tw:b:max2} if $g'=2$). This proves (2). 


If $n_1=7$, then \eqref{eq:pf:b:mx:ab} implies that $g'=3d$, which means that $g'$ satisfies the assumptions of (1).



If $n_1=9$, $n_1=11$, $n_1=17$, $n_1=23$, $n_1=33$ or $n_1=35$, then \eqref{eq:pf:b:mx:ab} and Lemma \ref{Lem:tw:b:max} imply that for some odd integer $k$, we have $g'=2^2\cdot 3k$, $g'=5k$, $g'=2^3k$, $g'=11k$, $g'=2^4\cdot 3k$ or $g'=17k$ respectively. Moreover, in each of these cases Lemmas \ref{Lem:tw:b:max} and \ref{Lem:tw:b:max2} give a data set of type \TypeB\ of the desired degree. 


On the other hand, if $n_1\in\{13,15,19,21,25,27,29,31,37,39\}$, then the situation is the same as in the case of $n_1=7$, that is there exists a data set of genus $g'$ with smaller $n_1$ (for $n_1\in\{15,27,39\}$ we use Lemma \ref{Lem:tw:b:max} to show that $g=0\mod 3$). This completes the proof of cases (4) -- (8) and it also proves that 
\begin{equation}\label{eq:pf:b:mx:inq} n\leq \frac{2g'}{1-\frac{1}{41}}=\frac{41}{20}g'\end{equation}
if $g'$ does not satisfy the assumptions of (1) -- (8). 

In fact, Lemmas \ref{Lem:tw:b:max} and \ref{Lem:tw:b:max2} give us a more general result:
\begin{itemize}	
	\item If $g'=l\cdot 3k$ for $l=1\mod 3$ and an odd integer $k$, then there exists a data set of type \TypeB\ with $g_0=0$, $m=1$ and \[n=\frac{2l+1}{l}\cdot g' =(2l+1)\cdot 3k=2g'+3k\]
	(apply Lemma \ref{Lem:tw:b:max} with $n_1=2l+1$ and $d=3k$).
	\item If $g'=l\cdot k$ for $l=2\mod 3$ and an odd integer $k$, then there exists a data set of type \TypeB\ with $g_0=0$, $m=1$ and \[n=\frac{2l+1}{l}\cdot g' =(2l+1)\cdot k=2g'+k\]
	(apply Lemma \ref{Lem:tw:b:max} with $n_1=2l+1$, $d=k$ if $k>1$, and Lemma~\ref{Lem:tw:b:max2} if $k=1$). Moreover, if $g'$ is prime, then $k=1$ and $n=2g'+1$ is the maximal possible degree of a root of $t_c$.
\end{itemize}
 This completes the proof of cases (9) and (10).

Note that applying Lemma \ref{Lem:tw:b:max} or Lemma \ref{Lem:tw:b:max2} with $n_1=2l+1$ for $l=0\mod 3$ yields $g'=l\cdot k = 3^s t \cdot k$ for some $s\geq 1$ and $t\neq 0 \mod 3$ which therefore satisfies (9), when $t=1 \mod 3$, and (10) when $t=2 \mod 3$. In particular, this implies that if $g'$ does not satisfy (1) -- (10), then $m\geq 2$.

Moreover, when $g'$ does not satisfy (1) -- (10) it must hold that $g'=1\mod 3$. That is the case, because:
\begin{itemize}
 \item $g'=2\mod 3$ satisfies the assumptions of (10) with $k=1$;
 \item $g'=3\cdot 2^s\cdot t$ for some odd integers $s$ and~$t$ satisfies (10) with $l=2^s$, $k=3t$;
 \item $g'=3\cdot 2^s\cdot t$ for an even integer $s$ and an odd integer $t$ satisfies the assumptions of (9) with $l=2^s$ and $k=t$. 
\end{itemize}

From now on we assume that $g'$ does not satisfy (1) -- (10), hence we concentrate on the case $m=2$ and $g'=1\mod 3$.
  
If $(n_1,n_2)=(3,3)$, then \eqref{eq:proof:b:max:23} implies that $g'=\frac{2}{3}n$ satisfies the assumptions of (2). Similarly, if $(n_1,n_2)=(3,5)$, then $g'=\frac{11}{15}n$ satisfies assumptions of (1), (4) or (6). 

For $(n_1,n_2)=(3,7)$ we obtain $g'=\frac{16}{21}n=16k$ for some odd integer $k=\frac{n}{21}$ and condition (D3B) yields 
\[b-a=ab\mod 3.\]
This implies that $a\neq b\mod 3$ and condition (D4B) gives
\[a+b+c_1\cdot \frac{n}{3}+c_2\cdot \frac{n}{7} = 0 \mod 3.\]
It follows that $c_1\cdot \frac{n}{3}=0\mod 3$, which proves that $k=\frac{n}{21}=0\mod 3$ (because $\gcd(c_1,n_1)=1$). Hence, $g'$ satisfies the assumptions of condition~(7).

Thus $(n_1,n_2)\not\in\{(3,3),(3,5),(3,7)\}$ and 
%
%
   \[2g'\geq \left(n-\frac{n}{5}\right)+\left(n-\frac{n}{5}\right)=\frac{8}{5}n,\]   
   which proves that $n\leq \frac{5}{4}g'$ for all data sets with $m=2$. 
   
   On the other hand, if $g'$ is odd, then there is a data set of type \TypeB:
   \[\left(g',1,(2,-2);\right),\]
   which gives a root of degree $g'$. This completes the proof of (11).
   
   Finally, if $g'$ does not satisfy (1) -- (11), then $g'=1\mod 3$ and $g'=0\mod 4$, hence $g'=4\mod 12$.
   In such a case, there is a data set 
   \[\left(g'+1,0,(2,-2);(1,g'+1),(-1,g'+1)\right),\]
   which gives a root of degree $g'+1$. This completes the proof of (12).
  \end{proof}
  \begin{uw}
   Using a computer it can be checked that there are 59 values of $g'\leq 500$ satisfying the assumptions of case (11) and for 31 of these values $N=g'$.
   
   As for case (12), there are 24 values of $g'\leq 500$ satisfying the assumptions of this case and only for four of these values: $g'\in \{4,16,64,256\}$ we have $N=g'+1$. The reason for the deficiency of genera with $N=g'+1$ follows from the fact that if $g'=2^lk$, where $l\geq 1$ and $k$ is odd, then there is a data set 
   \[\left(\frac{2^l+1}{2^l}\cdot g',0,(2,-2);(1,k),(-1,k)\right)\]
   of type \TypeB.
   This data set gives a root of degree $n>g'+1$ if $k>1$. There are 19 values of $g'\leq 500$ satisfying the assumptions of (12) for which $N>g'+1$ and for which the above data set gives a root of the maximal degree.
  \end{uw}
  \section{Primary roots of Dehn twists.}\label{sec:primary}
  As was observed by McCullough and Rajeevsarathy (Corollary 3.1 of~\cite{McCullough2010}), for each odd $n>1$ and $g+1>\frac{(n-2)(n-1)}{2}$ in the mapping class group $\mathcal{M}(S_{g+1})$ there exists a primary root of degree~$n$ of the Dehn twist $t_c$ about a nonseparating circle in $S_{g+1}$. As we will show below, a similar statement holds in the nonorientable case too.
  
  Following \cite{McCullough2010}, we say that a root of a Dehn twist is \emph{primary} if in the corresponding data set we have $n_i=n$ for $i = 1, \ldots, m$. In particular, if $n$ is prime, then all roots of degree $n$ are primary.
  \begin{tw}\label{tw:allrots:a}
   Let $c$ be a nonseparating two-sided circle in a nonorientable surface $N_{g+2}$ of genus $g+2$,
   such that $N_{g+2}\bez c$ is nonorientable, and let $n\geq 3$ be an odd integer.
   \begin{enumerate}
    \item If $g>(n-1)^2$, then there is a primary root of degree $n$ of the Dehn twist $t_c$.
    \item If $g=(n-1)^2$ or $g<n$, then there is no primary root of $t_c$ of degree~$n$.
   \end{enumerate} 
  \end{tw}
\begin{proof}
 Observe first that for each $g_0\geq 1$ and $m\geq 0$ there exists a primary data set of type \TypeA. In fact, we can always take $a=b=2$ and $c_i=1$, $i=1,\ldots,m$. 
 The genus of a primary data set of type \TypeA\ is given by the formula
 \[g=g_0n+m(n-1)=(g_0+m)(n-1)+g_0=(g_0+m)n_0+g_0,\]
 where $n_0=n-1$. 
 
 We collect all possible values of $g$ (for a fixed $n$) in a table, where subsequent rows correspond to values of $g_0\in\{1,2,\ldots,(n-1)\}$, and elements of each row increase according to increasing values of $m\in\{0,1,2,\ldots\}$. The first three columns of such a table are as follows.
 \[\begin{matrix}
    &n_0+1,& 2n_0+1,& 3n_0+1,\\
    &2n_0+2,& 3n_0+2,& 4n_0+2,\\
    &\ldots&\ldots&\ldots\\
    &(n-3)n_0+(n-3),& (n-2)n_0+(n-3),& \underline{(n-1)n_0+(n-3)},\\
    &(n-2)n_0+(n-2),& \underline{(n-1)n_0+(n-2)},& nn_0+(n-2),\\
    &\underline{(n-1) n_0+(n-1)},& nn_0+(n-1),&(n+1)n_0+(n-1).
   \end{matrix}
\]
It should be clear from the above table that for each $l\geq 1$ there exists a primary data set of type \TypeA\ with $g=(n-1)n_0+l$. Moreover, there is no primary data set with $g=(n-1)^2=(n-2)n_0+(n-1)$ and there are no primary data sets with $g<n_0+1=(n-1)+1=n$.
\end{proof}
 \begin{tw}\label{tw:allrots:b}
 Let $c$ be a nonseparating two-sided circle in a nonorientable surface $N_{2g'+2}$ of genus $2g'+2$,  $g'>0$,
 such that $N_{2g'+2}\bez c$ is orientable. Let $n\geq 3$ be an odd integer.
   \begin{enumerate}
    \item Assume that $n\neq 0\mod 3$. 
    \begin{itemize}
    \item[--] If $g'\geq \frac{(n-3)(n-1)}{2}$, then there is a primary root of degree $n$ of the Dehn twist $t_c$.
\item[--]  If $g'=\frac{(n-3)(n-1)}{2}-1$ or $g'<\frac{n-1}{2}$, then there is no primary root of $t_c$ of degree $n$.
\end{itemize}    
    \item Assume that $n=0\mod 3$. 
    \begin{itemize}
    \item[--] If $g'\geq \frac{(n-3)(n-1)}{2}$ and $g'\neq \frac{n^2-2n-1}{2}$, then there is a primary root of degree $n$ of $t_c$.
\item[--] If $g'=\frac{(n-3)(n-1)}{2}-1$, $g'=\frac{n^2-2n-1}{2}$ or $g'<n-1$, then there is no primary root of $t_c$ of degree $n$.
\end{itemize}
   \end{enumerate} 
  \end{tw}
  \begin{proof}
   Observe first that for each $g_0\geq 0$ and $m\geq 2$ we can construct a data set of type B with $(a,b)=(2,-2)$ and all $c_i\in\{-4,-2,-1,1\}$ (we choose them so that they satisfy (D4B)). There are also data sets with $m=0$ and $(a,b)=(2,-2)$ provided $g_0\geq 1$ ($g_0=0$ corresponds to the exceptional case of the Klein bottle), 
   whereas from Lemma \ref{Lem:tw:b:max2} we know that primary data sets of type \TypeB\ with $m=1$ exists if and only if $n\neq 0 \mod 3$.
   
    The genus of a primary data set of type \TypeB\ is given by the formula
 \[g'=g_0n+\frac{1}{2}m(n-1)=(2g_0+m)\cdot \frac{n-1}{2}+g_0=(2g_0+m)n_0+g_0,\]
 where $n_0=\frac{n-1}{2}$.
 
 We collect all possible values of $g'$ (for a fixed $n$) in a table, where subsequent rows correspond to values of $g_0\in\left\{0,1,2,\ldots,\frac{n-3}{2}\right\}$, and elements of each row increase according to increasing values of $m\in\{0,1,2,\ldots\}$. The first three columns of such a table are as follows.
 \[\begin{matrix}
    &0\cdot n_0+0,& n_0+0,& 2n_0+0,\\
    &2 n_0+1,& 3n_0+1,& 4n_0+1,\\
    &\ldots&\ldots&\ldots\\
    &(n-5)n_0+\frac{n-5}{2},& (n-4)n_0+\frac{n-5}{2},& \underline{(n-3)n_0+\frac{n-5}{2}},\\    
    &\underline{(n-3)n_0+\frac{n-3}{2}},& (n-2)n_0+\frac{n-3}{2},&(n-1)n_0+\frac{n-3}{2}.
   \end{matrix}
\]
It should be clear from the above table, that for each $l\geq 0$ and 
\[g'\neq (n-2)n_0+\frac{n-3}{2}=\frac{n^2-2n-1}{2}\]
there exists a primary data set of type \TypeB\ with $g=(n-3)n_0+l$. As for $g=\frac{n^2-2n-1}{2}$, a root of this degree exists if and only if $n\neq 0\mod 3$. Moreover, there is no primary data set with 
\[g'=(n-3)n_0-1=(n-4)n_0+\frac{n-3}{2},\] and there are no primary data sets with $g'<n_0=\frac{n-1}{2}$. There is also no data set with $g'=n_0$ if $n=0\mod 3$. 
  \end{proof}
  As was observed in \cite{McCullough2010} and \cite{Monden-roots}, Dehn twists about nonseparating circles in an orientable surface of genus $g+1\geq 2$ always have a root of degree 3 in $\mathcal{M}(S_{g+1})$. It turns out that a similar statement holds true also in the nonorientable case.
  \begin{wn}\label{wn:deg3}
   Let $c$ be a nonseparating two-sided circle in a nonorientable surface $N_g$. If a Dehn twist $t_c$ has a nontrivial root, then it has a root of degree 3.
  \end{wn}
  \begin{proof}
   Suppose first that $N_g\bez c$ is nonorientable. By Theorem \ref{tw:allrots:a}, if $g\geq 5$ then $t_c$ has a root of degree 3. Moreover, there is a data set of type \TypeA\ with $g_0=1$, $m=0$ and $(a,b)=(2,2)$ which defines a root of degree 3 of $t_c$ if $g=3$. By Theorem \ref{tw:root:existence}, this completes the proof in case \TypeA.
   
   If $N_g\bez c$ is orientable, then the proof is an immediate consequence of Theorems \ref{tw:root:existence} and \ref{tw:allrots:b}.     
  \end{proof}
\section{Expressions for primary roots.}\label{sec:geom:express}
In this section we construct primary roots of Dehn twists of all possible degrees. To be more precise, recall from the previous section, that the genus of a primary data set of type A and degree $n$ is given by the formula
\begin{equation}\label{eq:genusA_n}
 g=g_0n+m(n-1), 
\end{equation}
where $g_0 \geq 1$ and $m \geq 0$. Similarly, if $g=2g'$ is the genus of a primary data set of type B and degree $n$, then 
\begin{equation}\label{eq:genusB_n}
 g'=g_0n+\frac{1}{2}m(n-1), 
\end{equation}
where $g_0\geq 0$ and $m\geq 0$. Our goal is to explicitly construct examples of such roots for all possible values of $n$, $g_0$ and $m$. By an 'explicit construction' we mean an expression as a product of Dehn twists on $N_{g+2}$ --- note that by Proposition \ref{Prop:odd:degree}, each root $h$ of a Dehn twist on $N_{g+2}$ is a product of twists, that is, $h$ is an element of the \emph{twist subgroup} $\mathcal{T}(N_{g+2})$ of $\mathcal{M}(N_{g+2})$.

In the oriented case, for $n=3$, a similar construction was carried out by Monden \cite{Monden-roots}. His construction is based on the star and the chain relations, but it turns out that these relations are not enough if we want to construct roots with $n>3$. In such a case we need more general relations and this is exactly the reason for introducing the trident relation (Proposition \ref{tw:trident}).

%
%
\subsection*{Roots of Dehn twists of type $\bm{A}$}
Fix $g_0\geq 1$, $m\geq 0$, an odd integer $n\geq 3$ and consider a surface $N$ which is decomposed as a sum of $g_0$ nonorientable surfaces $N_{n,2}$ of genus $n$ with two boundary components and $m$ orientable surfaces $S_{\frac{n-1}{2},2}$ of genus $\frac{n-1}{2}$ with two boundary components (Figure \ref{fig:A_decomp_n}). 
\begin{figure}[h]
			\begin{center}
				\includegraphics[width=1\textwidth]{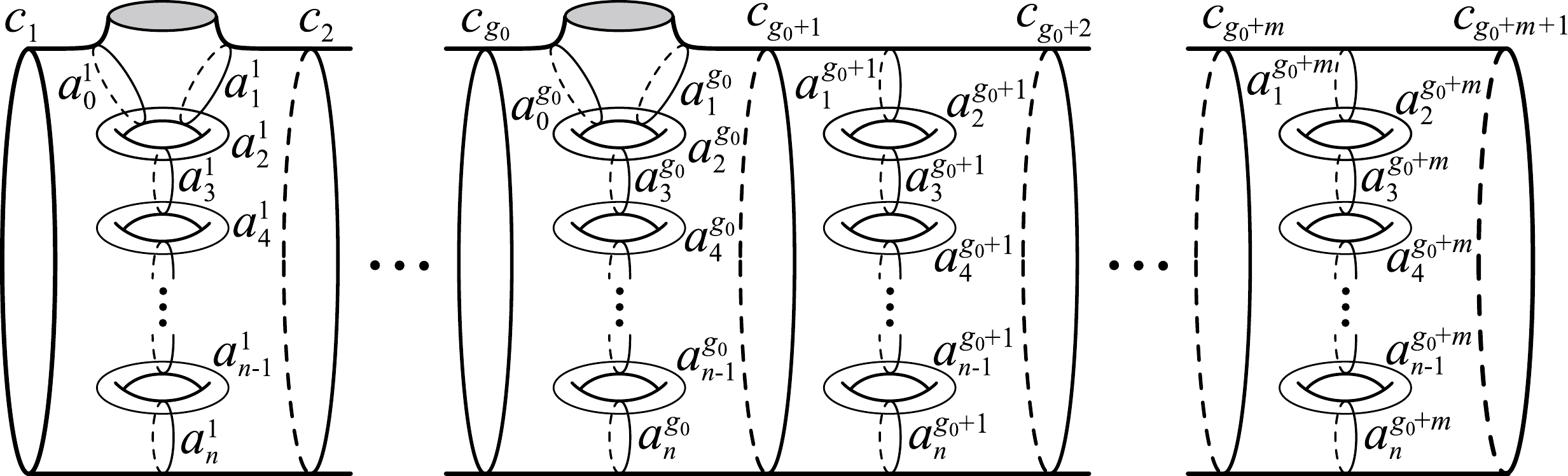}
				\caption{A decomposition of $N_{g+2} \bez c = N_{g,2}$ into $g_0$ subsurfaces $N_{n,2}$ and $m$ subsurfaces $S_{\frac{n-1}{2},2}$.}
				\label{fig:A_decomp_n} 
			\end{center}
		\end{figure} 
		
By \eqref{eq:genusA_n}, $N$ is a nonorientable surface of genus $g$ and has two boundary components. Let 
\[\begin{aligned}
&c_1,c_2,\ldots,c_{g_0},c_{g_0+1},\ldots,c_{g_0+m},c_{g_0+m+1},\\
&a_{0}^i,a_{1}^i,a_2^i,\ldots,a_{n}^i,\text{ for $i=1,\ldots,g_0$,}\\
&a_1^i,a_2^i,\ldots,a_{n}^i,\text{ for $i=g_0+1,\ldots,g_0+m$}
\end{aligned}\]
be circles as in Figure \ref{fig:A_decomp_n}, that is $c_i$ and $c_{i+1}$ bound a nonorientable subsurface $N_{n,2}$ of genus $n$ if $i=1,\ldots,g_0$, and they bound an oriented subsurface $S_{\frac{n-1}{2},2}$ of genus $\frac{n-1}{2}$ if $i=g_0+1,\ldots,g_0+m$.
Let 
\[\begin{aligned}
\delta_i&=t_{{a_0^{i}}}t_{{a_1^{i}}}t_{{a_2^{i}}}\ldots t_{{a_{n}^{i}}},\text{ for $i=1,\ldots,g_0$,}\\
\delta_i&=t_{a_1^{i}}^2t_{a_2^{i}}t_{a_3^{i}}\ldots t_{a_{n}^{i}}, \text{ for $i=g_0+1,\ldots,g_0+m$.}
  \end{aligned}
\]
By Proposition \ref{chain} and Corollary \ref{wn:trident},
\[\delta_i^n=t_{c_{i}}t_{c_{i+1}},\text{ for $i=1,\ldots,g_0+m$,}\]
Hence, if 
\[\zeta=\begin{cases}
     \delta_1\delta_2^{-1}\delta_3\delta_4^{-1}\cdots \delta_{g_0+m-1}\delta_{g_0+m}^{-1} &\text{if $g_0+m$ is even}\\
     \delta_1\delta_2^{-1}\delta_3\delta_4^{-1}\cdots \delta_{g_0+m-1}^{-1}\delta_{g_0+m} &\text{if $g_0+m$ is odd,}
    \end{cases}
\]
then
\[\zeta^n=\begin{cases}
          t_{c_1}t_{c_{g_0+m+1}}^{-1}&\text{if $g_0+m$ is even}\\
          t_{c_1}t_{c_{g_0+m+1}}&\text{if $g_0+m$ is odd.}
         \end{cases}
\]
Let $N_{g+2}$ be a closed surface obtained from $N=N_{g,2}$ by gluing $c_1$ to $c_{g_0+m+1}$ with such a choice of orientations that the element $\xi\in\mathcal{M}(N_{g+2})$ induced by $\zeta\in\mathcal{M}(N_{g,2})$ satisfies $\xi^n=t_c^2$, where $c$ is a circle in $N_{g+2}$ obtained from~$c_1$ in $N_{g,2}$. Since $\xi$ commutes with $t_c$, we obtain that  
\[h=\xi^{\frac{n+1}{2}}t_c^{-1}\]
is a desired root of degree $n$ of $t_c$. 
\begin{uw}
From the proof of Proposition \ref{tw:trident} it follows that for each $i=g_0+1,\ldots,g_0+m$, $\delta_i$ has one fixed point in the corresponding subsurface $S_{\frac{n-1}{2},2}$ around which it locally acts like a rotation through an angle \mbox{$ \frac{2\pi}{n}\cdot \frac{n-1}{2}$} (this can also be deduced from \cite{MargSchleim}). Furthermore, for each $i=1,\ldots,g_0$, $\delta_i$ has no fixed points in the corresponding subsurface $N_{n,2}$. Hence, 
$\zeta$ has $m$ fixed points in $N_{g,2}$ 
and locally around them it acts like a rotation through an angle $\pm \frac{2\pi}{n}\cdot \frac{n-1}{2}$. Therefore $h$ constructed above corresponds to the data set
\[(n, g_0, (2,\pm2); (\pm4,n),(\pm4,n), \ldots, (\pm4,n)).\]
\end{uw}

%
%

\begin{uw}
Let us mention that there is a different natural construction of a root of $t_c$ which has equivalent data set to the data set of a root $h$ defined above. We start from the same decomposition of $N$ as a connected sum of $g_0$ nonorientable surfaces $N_{n,2}$ of genus $n$ and $m$ orientable surfaces $S_{\frac{n-1}{2},2}$ of genus $\frac{n-1}{2}$, but now we use a different model for surfaces $N_{n,2}$ -- we represent each of them as an annulus with $n$ crosscaps $\mu_1^i,\mu_2^i,\ldots,\mu_n^i$, $i=1,2,\ldots,g_0$ \linebreak (Figure \ref{fig:A_decomp_n2}).
 \begin{figure}[h]
			\begin{center}
				\includegraphics[width=1\textwidth]{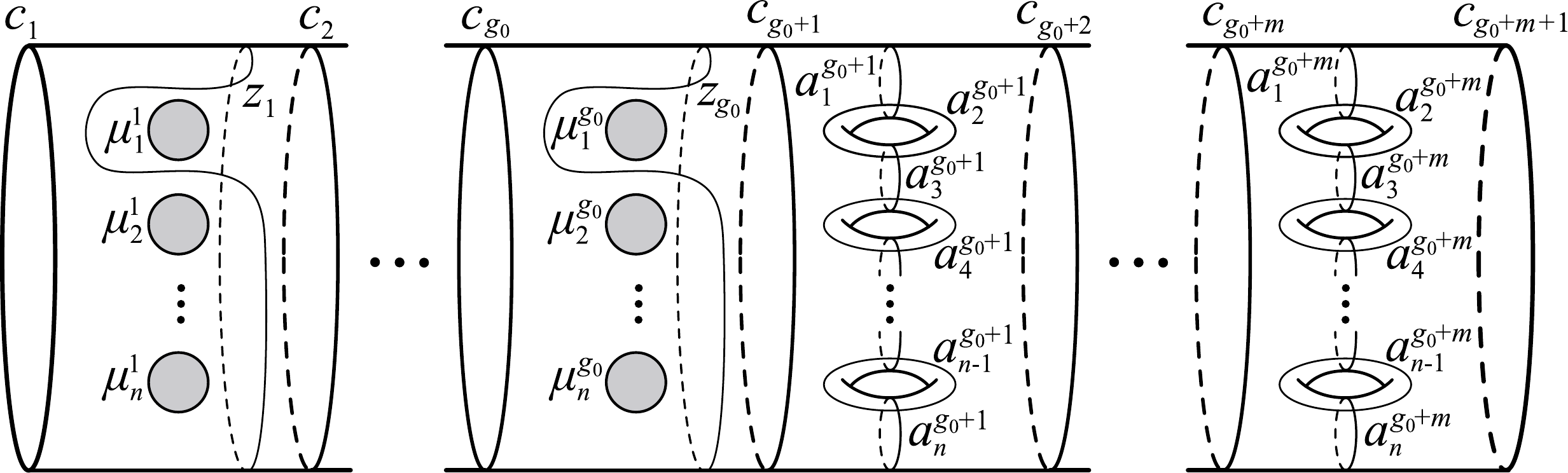}
				\caption{A decomposition of $N_{g+2} \bez c = N_{g,2}$ into $g_0$ subsurfaces $N_{n,2}$ and $m$ subsurfaces $S_{\frac{n-1}{2},2}$.}
				\label{fig:A_decomp_n2} 
			\end{center}
		\end{figure} 
Let $u_{j,i}$ be a \emph{crosscap transposition} of $\mu_j^i$ and $\mu_{j+1}^i$, $i=1,2,\ldots,g_0$, $j=1,2,\ldots,n-1$ (for a definition see for example Section 2 of \cite{ParlakStukowRoots}) and let
\[\eta_i=t_{c_{i+1}}^{-1}t_{z_i}u_{1,i}u_{2,i}\cdots u_{n-1,i}, \text{ for $i=1,\ldots,g_0$,}\]
where $z_i$ is as in Figure \ref{fig:A_decomp_n2}. Geometrically, $\eta_i$ is a natural cyclic rotation of $n$ crosscaps on $N_{n,2}$, so it does not have any fixed points and
\[\eta_i^n=t_{c_i}t_{c_{i+1}}^{-1}, \text{ for $i=1,\ldots,g_0$,}.\]
Thus if 
\[\zeta=\begin{cases}
     \eta_1\eta_2\cdots\eta_{g_0}\delta_{g_0+1}\delta_{g_0+2}^{-1}\delta_{g_0+3}\cdots  \delta_{g_0+m-1}\delta_{g_0+m}^{-1} &\text{if $m$ is even}\\
     \eta_1\eta_2\cdots\eta_{g_0}\delta_{g_0+1}\delta_{g_0+2}^{-1}\delta_{g_0+3}\cdots  \delta_{g_0+m-1}^{-1}\delta_{g_0+m} &\text{if $m$ is odd,}
    \end{cases}
\]
then
\[\zeta^n=\begin{cases}
          t_{c_1}t_{c_{g_0+m+1}}^{-1}&\text{if $m$ is even}\\
          t_{c_1}t_{c_{g_0+m+1}}&\text{if $m$ is odd}
         \end{cases}
\]
and we can complete the construction of a root as before.

In particular, when $g=n$ is odd (hence $g_0=1$, $m=0$), then
Theorem~\ref{tw:orb} implies that the natural cyclic rotation $u_{1,1}u_{2,1}\cdots u_{g-1,1}$ of $g$ crosscaps on~$N_g$ is conjugate to $t_{a_0^1}t_{a_1^1}t_{a_2^1}\cdots t_{a_g^1}$ in $\mathcal{M}(N_g)$.
\end{uw}
\subsection*{Roots of Dehn twists of type $\bm{B}$}
Fix $g_0\geq 0$, $m\geq 0$, an odd integer $n\geq 3$ and consider a surface $S$ decomposed as a sum of $g_0$ oriented surfaces $S^1,S^2,\ldots,S^{g_0}$ of genus $n$ with two boundary components and $m$ oriented surfaces $S^{g_0+1},S^{g_0+2},\ldots, S^{g_0+m}$ of genus $\frac{n-1}{2}$ with two boundary components (Figure \ref{fig:B_decomp_n}). 
\begin{figure}[h] \begin{center}
				\includegraphics[width=1\textwidth]{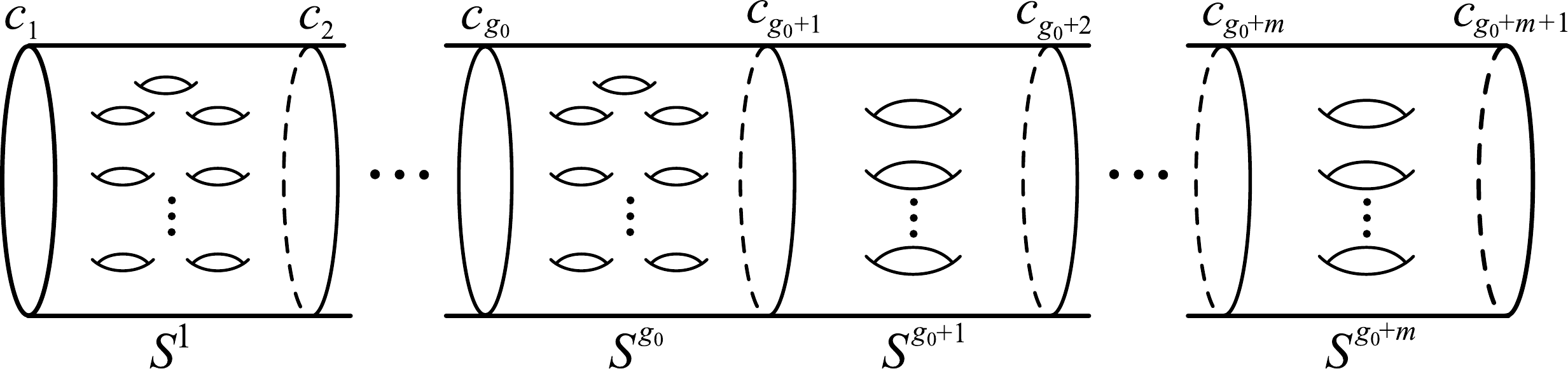}
				\caption{A decomposition of $N_{g+2} \bez c = N_{g,2}$ into $g_0$ subsurfaces $S_{n,2}$ and $m$ subsurfaces $S_{\frac{n-1}{2},2}$.}
				\label{fig:B_decomp_n} 
			\end{center}
		\end{figure} 
By \eqref{eq:genusB_n}, $S$ is an orientable surface of genus $g'$ and has  two boundary components. Let 
\[c_1,c_2,\ldots,c_{g_0},c_{g_0+1},\ldots,c_{g_0+m},c_{g_0+m+1}\]
be circles as in Figure \ref{fig:B_decomp_n}, that is $c_i$ and $c_{i+1}$ bound an oriented subsurface $S^i$ of genus $n$ if $i=1,\ldots,g_0$, and they bound an oriented subsurface $S^i$ of genus $\frac{n-1}{2}$ if $i=g_0+1,\ldots,g_0+m$.

By Corollary \ref{wn:trid:2}, for each $i=1,\ldots,g_0$, there exists a product $\delta_i$ of Dehn twists with support in $S^i$ and such that 
\[\delta_i^{n}=t_{c_i}t_{c_{i+1}}^{-1}.\]
Similarly, from Proposition \ref{chain} if follows that for each $i=g_0+1,\ldots,g_0+m$ there exists a product $\delta_i$ of Dehn twists with support in $S^i$ such that
\[\delta_i^{n}=t_{c_i}t_{c_{i+1}}.\]

We will need two more products $\delta$ and $\theta$ of Dehn twists. The first one is supported in $S^{g_0+m-1}\cup S^{g_0+m}$ and satisfies
\[\delta^n=t_{c_{g_0+m-1}}t_{c_{g_0+m+1}}\]
and the second one is supported in $S^{g_0+m}$ and satisfies
\[\theta^n=t_{c_{g_0+m}}t_{c_{g_0+m+1}}^{\frac{n-1}{2}}.\]
The existence of $\delta$ is a straightforward consequence of Corollary \ref{wn:trid:3} and to construct $\theta$ take $n=2g+1$ and glue a disk to $c_3$ in the statement of Proposition \ref{tw:trident}.

Thus if 
\[\zeta=\begin{cases}
        \delta_1\delta_2\cdots\delta_{g_0}\delta_{g_0+1}\delta_{g_0+2}^{-1}\delta_{g_0+3}\cdots \delta_{g_0+m-1}\delta_{g_0+m}^{-1}&\text{if $m\geq 0$ is even}\\
        \delta_1\delta_2\cdots\delta_{g_0}\delta_{g_0+1}\delta_{g_0+2}^{-1}\delta_{g_0+3}\cdots \delta_{g_0+m-3}^{-1}\delta_{g_0+m-2}\delta^{-1}&\text{if $m\geq 3$ is odd}\\
        \delta_1\delta_2\cdots\delta_{g_0}\theta&\text{if $m=1$,}\\
       \end{cases}
\]
then 
\[\zeta^n=\begin{cases}
          t_{c_1}t_{c_{g_0+m+1}}^{-1}&\text{if $m\neq 1$}\\
          t_{c_1}t_{c_{g_0+m+1}}^{\frac{n-1}{2}}&\text{if $m=1$.}
         \end{cases}
\]
Let $N_{2g'+2}=N_{g+2}$ be a closed surface obtained from $S=S_{g',2}$ by gluing $c_1$ to $c_{g_0+m+1}$ with such a choice of orientations that the resulting surface is nonorientable. The element $\xi\in\mathcal{M}(N_{g+2})$ induced by $\zeta\in\mathcal{M}(S_{g',2})$ satisfies 
\[\xi^n=\begin{cases}
       t_c^2&\text{if $m\neq 1$}\\
       t_c^{\frac{3-n}{2}}&\text{if $m=1$,}
      \end{cases}
\]
where $c$ is a circle in $N_{g+2}$ obtained from $c_1$ in $S_{g',2}$. Moreover $\xi$ commutes with $t_c$ and thus 
\[h=\begin{cases}
\xi^{\frac{n+1}{2}}t_c^{-1}&\text{if $m\neq 1$}\\
\xi^{\frac{n+2}{3}}t_c^{\frac{n-1}{6}}&\text{if $m=1$ and $n=1\mod 3$}\\
\xi^{\frac{2n+2}{3}}t_c^{\frac{n-2}{3}}&\text{if $m=1$ and $n=2\mod 3$}
    \end{cases}
\]
is a desired root of degree $n$ of $t_c$. Corresponding data sets are as follows
%
\[\begin{cases}
(n,g_0,(2,-2);(-4,n),(4,n),(-4,n),\ldots,(-4,n), (4,n))&\text{\begin{tabular}{@{}l@{}}if $m \geq 0$ \\ is even,\end{tabular}}\\
(n,g_0,(2,-2);(-4,n),(4,n),(-4,n),\ldots,(-4,n), (2,n),(2,n))&\text{\begin{tabular}{@{}l@{}}if $m \geq 3$ \\ is odd,\end{tabular}}\\
\left(n,g_0,\left(\frac{n+3}{2},-3\right);\left(\frac{n+3}{2},n\right)\right) &\text{\begin{tabular}{@{}l@{}}if $m = 1$. \\\end{tabular}}\\
\end{cases}
\]

As for the missing case of $m=1$ and $n=0\mod 3$, recall that we observed in the proof of Theorem \ref{tw:allrots:b} that there are no such primary roots of type B (see also Lemma \ref{Lem:tw:b:max2}).

\end{document}